%% file: 2022aaai_final.tex
\documentclass[dvipsnames,letterpaper]{article} %
\usepackage{fullpage}
\usepackage[hyphens]{url}  %
\usepackage{graphicx} %
\usepackage{xcolor}
\usepackage{caption} %
\DeclareCaptionStyle{ruled}{labelfont=normalfont,labelsep=colon,strut=off} %
\usepackage{algorithmic}
\usepackage[ruled,vlined]{algorithm2e}
\usepackage{hyperref}
\hypersetup{
    colorlinks=true,       %
    linkcolor=Red,          %
    citecolor=Blue,        %
    filecolor=Cyan,         %
    urlcolor=magenta        %
}
\usepackage[round,sort]{natbib}

\usepackage{xcolor}
\usepackage{amsmath,amssymb,mathtools,amsthm, thmtools}
\usepackage{enumitem}
\usepackage{thm-restate}
\newif\ifapp
\appfalse
\newif\ifapp
\appfalse

\title{Decision-Dependent Risk Minimization in Geometrically Decaying Dynamic Environments}
\author{
   Mitas Ray\thanks{mitasray@uw.edu} \thanks{Department of Electrical and Computer Engineering, University of Washington, Seattle} ,
    Dmitriy Drusvyatskiy\thanks{Department of Mathematics, University of Washington, Seattle} ,
    Maryam Fazel\footnotemark[2] ,
   Lillian J. Ratliff\footnotemark[2] %
}

\newtheorem{lemma}{Lemma}

\newtheorem{assumption}{Assumption}
\newtheorem{definition}{Definition}

\newtheorem{corollary}{Corollary}

\newcommand{\mc}{\mathcal}
\newcommand{\mb}{\mathbb}
\newcommand{\ra}{\rangle}
\newcommand{\la}{\langle}

\DeclareMathOperator*{\E}{\mathbb{E}}
\DeclareMathOperator*{\argmin}{argmin}
\DeclareMathOperator*{\proj}{proj}

\newcommand{\X}{\mc{X}}
\newcommand{\Z}{\mc{Z}}
\newcommand{\LL}{\mc{L}}
\newcommand{\dimm}{d}
\newcommand{\dist}{p}
\newcommand{\georate}{\lambda}
\newcommand{\epoch}{n}

\newcommand{\bdelta}{\bar{x}^\delta}
\newcommand{\LLlip}{G}
\newcommand{\op}{\mathrm{op}}
\newcommand{\bmat}[1]{\begin{bmatrix}#1\end{bmatrix}}
\begin{document}

\maketitle

\begin{abstract}
This paper studies the problem of expected loss minimization given a data distribution that is dependent on the decision-maker's action and evolves dynamically in time according to a geometric decay process. 
Novel algorithms for both the information setting in which the decision-maker has a first order gradient oracle and the setting in which they have simply a loss function oracle are introduced. The algorithms operate on the same underlying principle: the decision-maker repeatedly deploys a fixed decision over the length of an epoch, 
thereby allowing the dynamically changing environment to sufficiently mix before updating the decision. 
The iteration complexity in each of the settings is shown to match existing rates for first and zero order stochastic gradient methods up to logarithmic factors.
The algorithms are evaluated on a ``semi-synthetic" example using real world data from the SFpark dynamic pricing pilot study; it is shown that the  announced prices result in an improvement for the institution's objective (target occupancy), while achieving an overall reduction in parking rates.
\end{abstract}

\section{Introduction}
\label{sec:intro}
\input{intro}

\section{Preliminaries}
\label{sec:prelims}

\input{prelims}

\section{Algorithms \& Sample Complexity Analysis}
\label{sec:algs}
\label{sec:results}

As alluded to in the introduction, the algorithms we propose for each of the information settings are similar in spirit: 
they each operate in epochs by holding fixed a decision for $\epoch$ steps and querying the environment until the distribution dynamics have mixed sufficiently towards the fixed point distribution corresponding to the current decision.

\begin{algorithm}[t!]
\SetAlgoLined
 \textbf{Initialization}: epoch length $\epoch_t$, step-size $\eta_t=\frac{4}{t\alpha}$, initial point $x_1$, query radius $\delta$, horizon $T$, initial distribution $\dist_0$\;
\For{$t=1,2,\ldots, T$ }{
\textcolor{blue}{\small \texttt{// Step 1: Query-Mix}} \\
Sample vector $v_t$ from the unit sphere\;
Query with $x_t+\delta v_t$ for $\epoch_t$ steps, so that $\dist_t=\mc{T}^{\epoch_t}(\dist_{t-1},x_t+\delta v_t)$\;
\textcolor{blue}{\small\texttt{// Step 2: Update}} \\
Oracle reveals $\hat{g}_t=\frac{\dimm}{\delta}\ell(z,x_t+\delta v_t)v_t$, $z\sim \dist_t$\;
Update $x_{t+1}=\proj_{(1-\delta){\X}}(x_t-\eta_t\hat{g}_t)$\;
}
 \caption{Epoch-Based Zeroth Order Algorithm}
 \label{alg:main-zo}
\end{algorithm}
\begin{algorithm}[t!]
\SetAlgoLined
 \textbf{Initialization}: epoch length $\epoch_t$, non-increasing step-size $\eta_t$, initial point $x_1$,
 horizon $T$, initial distribution $\dist_0$\;
\For{$t=1,2,\ldots, T$ }{
\textcolor{blue}{\small \texttt{// Step 1: Query-Mix}} \\
Query with ${x}_t$ for $\epoch_t$ steps, so that $\dist_t=\mc{T}^{\epoch_t}(\dist_{t-1},x_t)$\;

\textcolor{blue}{\small\texttt{// Step 2: Update}} \\
Oracle reveals $\hat{g}_t=\nabla\ell(x_t,z)$, $z\sim p_t$\;
Update $x_{t+1}=\proj_{{\X}}(x_t-\eta_t\hat{g}_t)$\;
}
 \caption{Epoch-Based First Order Algorithm}
 \label{alg:main-fo}
\end{algorithm}

\label{subsec:results_empirical}

\input{results_empirical}

\section{Numerical Experiments}
\label{sec:experiments}
\input{experiments}

\section{Discussion and Future Directions}

This work is an important step in understanding performative prediction in dynamic environments. Moving forward there are a number of interesting future directions. In this work, we consider one class of well-motivated dynamics. Another practically motivated class of dynamics are period dynamics; indeed, in many applications there is an external context which evolves periodically such as seasonality or other temporal effects. Devising algorithms for such cases is an interesting direction of future work.
As compared to classical reinforcement learning problems, in this work, we exploit the structure of the dynamics along with convexity to devise convergent algorithms. However, we only considered  general conditions on the class of distributions $\mc{D}(x)$; it may be possible to exploit additional structure on $\mc{D}(x)$ in improving the sample complexity of the proposed algorithms or devising more appropriate algorithms that leverage this structure. 
\bibliographystyle{plainnat}
\bibliography{2022aaai-refs}
\newpage
\onecolumn
\appendix

\section{Technical Lemmas and Notation}
\label{app:techlemmas}

\paragraph{Notation.} Throughout we will use the following derivative and partial derivative notation. For a given function $\ell(x,z)$, the partial derivative of $\ell$ with respect to $z$ is denoted $\nabla_z\ell(x,z)$ and the partial derivative with respect to $x$ is denoted $\nabla_x \ell(x,z)$. For the expected risk $\mb{E}_{z\sim \mc{D}(x)}[\ell(x,z)]$, the total derivative with respect to $x$ is denoted 
\begin{align*}
    \nabla \E_{z\sim \mc{D}(x)}[\ell(x,z)]&=\nabla\left(\int \ell(x,z)\boldsymbol{p}_x(z)dz\right)=\E_{z\sim \mc{D}(x)}[\nabla_x\ell(x,z)]+\E_{z\sim \mc{D}(x)}[\ell(x,z)\nabla_x \log(\boldsymbol{p}_x(z))]
\end{align*}
where $\boldsymbol{p}_x(z)$ is the density function for $\mc{D}(x)$ and in the last equality we have applied the so-called `log trick', which comes from the chain rule.
Throughout, we use the notation $\|\cdot\|$ for the Euclidean norm.

\paragraph{Technical Lemmas.}
The following lemma is a direct consequence of 
dual form of the Wasserstein-$1$ distance.
\begin{lemma}
\label{lemma:kr_duality2}
Let $f:\mb{R}^n\rightarrow\mb{R}^n$ be $\beta$-Lipschitz, and let $X,X'$ be random vectors with distributions $p$ and $p'$, respectively. 
Then,
\[\|\mb{E}[f(X)]-\mb{E}[f(X')]\|\leq\beta{W}_1(p,p').\]
\end{lemma}

\lemmaparametric*
\begin{proof}
Observe that we may write
\[\nabla \LL(x)=\E_{\zeta\sim \mc{P}}V^\top \nabla_{x,z}\ell(z,\zeta+Ax)\quad \text{where}\ V=\bmat{I & 0 \\ 0 & A}.\]
Therefore, we deduce
\begin{align*}
    \|\nabla \LL(x)-\nabla \LL(y)\|&\leq \|V\|_{\op}\E_{\zeta\sim \mc{P}}\|\nabla_{x,z}\ell(z,\zeta+Ax)-\nabla_{x,z}\ell(y,\zeta+Ay)\|\\
    &\leq \max\{1,\|A\|_{\op}\}\cdot\beta\cdot\E_{\zeta\sim \mc{P}}\|(z,\zeta+Ax)-(y,\zeta+Ay)\|\\
    &=\max\{1,\|A\|_{\op}\}\cdot\beta\cdot\sqrt{\|x-y\|^2+\|A(x-y)\|^2}\\
    &\leq \max\{1,\|A\|_{\op}\}\cdot\beta\cdot\sqrt{(1+\|A\|_{\op}^2)}\cdot\|x-y\|.
\end{align*}
Analogously, observe that
\[\nabla^2 \LL(x)=\E_{\zeta\sim \mc{P}}V^\top \nabla^2\ell(x,\zeta+Ax) V\]
where \[\nabla^2\ell(x,z)=\bmat{\nabla^2_x\ell(x,z) & \nabla_{xz}\ell(x,z)\\ \nabla_{zx}\ell(x,z) & \nabla^2_z\ell(x,z)}\] is the Hessian of $\ell$ with respect to $(x,z)$.
Therefore, we deduce
\begin{align*}
    \|\nabla^2 \LL(x)-\nabla^2 \LL(y)\|&\leq \|V\|_{\op}^2\E_{\zeta\sim \mc{P}}\|\nabla^2\ell(z,\zeta+Ax)-\nabla^2\ell(y,\zeta+Ay)\|\\
    &\leq \max\{1,\|A\|_{\op}^2\}\cdot\rho\cdot\E_{\zeta\sim \mc{P}}\|(z,\zeta+Ax)-(y,\zeta+Ay)\|\\
    &=\max\{1,\|A\|_{\op}^2\}\cdot\rho\cdot\sqrt{\|x-y\|^2+\|A(x-y)\|^2}\\
    &\leq \max\{1,\|A\|_{\op}^2\}\cdot\rho\cdot\sqrt{(1+\|A\|_{\op}^2)}\cdot\|x-y\|.
\end{align*}
The proof is complete.
\end{proof}

\section{Proofs for Zero Order Oracle Setting}
\label{app:zoproofs}
\label{app:zero_order}

\subsection{Technical Lemmas}

Recall that 
\[{\LL}^\delta(x)=\mb{E}_{v\sim \mb{B}}[\mb{E}_{z\sim\mc{D}(x+\delta v)}[\ell(z,x+\delta v)]].\]
\begin{restatable}{lemma}{lemmastrongconvexmu}
Suppose that Assumptions~\ref{a:standing}, \ref{a:mixture_dominance}, and \ref{a:smooth_ell} hold.  Choose $\delta\leq c\alpha/H$ for some constant $c\in (0,1)$. Then the map $\LL^\delta$ is strongly convex over $\X$ with parameter $(1-c)\alpha$.
\label{lem:strong_convex_mu}
\end{restatable}
\begin{proof}
We first estimate the Lipschitz constant of the difference map
\[h(x):=\nabla \LL^\delta(x)-\nabla \LL(x).\]
To this end, we compute
\[\nabla h(x)=\mb{E}_{w\sim \mb{B}}[\nabla^2\LL(x+\delta w)-\nabla^2\LL(x)].\]
Taking into account that the map $x\mapsto \nabla^2\LL(x)$ is $H$-Lipschitz continuous, we deduce
\[\|\nabla h(x)\|_{\text{op}}\leq \mb{E}_{w\sim \mb{B}}[\|\nabla^2\LL(x+\delta w)-\nabla^2\LL(x)\|_{\text{op}}]\leq \delta H\mb{E}_{w\sim \mb{B}}\|w\|\leq \delta H.\]
Thus the map $h$ is Lipschitz continuous with parameter $\delta H$. We therefore compute
\begin{align*}
    \la \nabla \LL^\delta(x)-\nabla \LL^\delta(x'),x-x'\ra&=\la \nabla \LL(x)-\nabla \LL (x'),x-x'\ra+\la h(x)-h(x'),x-x'\ra\geq (\alpha-H\delta)\|x-x'\|^2,
\end{align*}
which completes the proof.
\end{proof}

\label{app:proof_lemma_expected_bounded_gradient}

\lemmaexpectedboundedgradient*

\begin{proof}[Proof of Lemma~\ref{lemma:expected_bounded_gradient}]
Observe that 
using Jensen's inequality along with Lemma~\ref{lemma:kr_duality2},  we deduce
\begin{align*}
    \|\nabla \mb{E}_{v\sim\mb{B}}[\mb{E}_{z\sim \dist_t}[\ell(z,x_t+\delta v)]]-\nabla {\LL}^\delta(x_t)\|^2&\leq \mb{E}_{v\sim\mb{B}}\left[\|\nabla \mb{E}_{z\sim \dist_t}[\ell(z,x_t+\delta v)]-\nabla \mb{E}_{z\sim \mc{D}(x_t+\delta v)}[\ell(z,x_t+\delta v)]\|^2 \right]\\
    &\leq \mb{E}_{v\sim\mb{B}}[L^2({W}_1(\dist_t,\mc{D}(x_t+\delta v)))^2].
\end{align*}
Hence, we need an 
 an upper bound on $\mc{W}_1(\dist_{t},\mc{D}(x_t+\delta v_t))$ which is the Wasserstein-1 distance between the distribution at time $t$ and the fixed point distribution for the query point $x_t+\delta v_t$. 

\paragraph{Upper bound on ${W}_1(\dist_{t},\mc{D}(x_t+\delta v))$.}
Using the fact that $\dist_t=\georate^{\epoch_t} \dist_{t-1}+(1-\georate^{\epoch_t})\mc{D}(x_t+\delta v)$, we expand ${W}_1(\mc{D}(x_t+\delta v), \dist_{t})$ as follows:
\begin{align*}
{W}_1(\dist_{t},\mc{D}(x_t+\delta v))
&={W}_1(\georate^{\epoch_{t}} \dist_{t-1}+(1-\georate^{\epoch_{t}})\mc{D}(x_{t}+\delta v),\mc{D}(x_t+\delta v))
\\&\leq\georate^{\epoch_{t}}{W}_1(\dist_{t-1},\mc{D}(x_t+\delta v) )+(1-\georate^{\epoch_{t}}
){W}_1(\mc{D}(x_t+\delta v),\mc{D}(x_{t}+\delta v))\\
&=\georate^{\epoch_{t}}{W}_1(\dist_{t-1},\mc{D}(x_t+\delta v) )\\
&=\georate^{\epoch_{t}}{W}_1(\georate^{n_{t-1}}\dist_{t-2}+(1-\georate^{n_{t-1}})\mc{D}(x_{t-1}+\delta v),\mc{D}(x_t+\delta v) )\\
&\leq \georate^{\epoch_{t}}\cdot \georate\cdot W_1(p_{t-2},\mc{D}(x_t+\delta v))+\georate^{\epoch_{t}}(1-\georate^{n_t})W_1(\mc{D}(x_{t-1}+\delta v),\mc{D}(x_t+\delta v) )
\\
&\leq\georate^{\epoch_{t}}\cdot\lambda\cdot {W}_1(\dist_{t-2},\mc{D}(x_t+\delta v))+\georate^{\epoch_{t}}(1-\georate^{\epoch_{t}})\cdot\gamma\cdot\|x_t-x_{t-1}\|\\
&\leq\georate^{\epoch_{t}}\cdot\lambda\cdot {W}_1(\dist_{t-2},\mc{D}(x_t+\delta v))+\georate^{\epoch_{t}}\cdot\gamma\cdot\|x_t-x_{t-1}\|,
\end{align*}
where we have used the triangle inequality, Assumption \ref{a:standing}(d), and the fact that $\lambda>\lambda^{n_t}$ for any $t\geq 1$, $\lambda^{n_{t}}<\lambda^{n_{t-i}}$ for any $t\in\{1,\ldots, t-1\}$, and $1-\lambda^{n_t}<1$.
Continuing to unroll the recursion, we have that
\begin{align}
\mb{E}_v{W}_1(\dist_{t},\mc{D}(x_t+\delta v))
&\leq\georate^{\epoch_t}\cdot\georate\mb{E}_v{W}_1(\georate^{\epoch_{t-2}} \dist_{t-3}+(1-\georate^{\epoch_{t-2}})\mc{D}(x_{t-2}+\delta v_{t-2}),\mc{D}(x_t+\delta v_t))+\georate^{\epoch_{t}}\cdot\gamma\cdot \|x_t-x_{t-1}\|\notag
\\&\leq\georate^{\epoch_t}\georate^{2}\mb{E}_vW_1(\dist_{t-3},\mc{D}(x_t+\delta v))+\georate\georate^{\epoch_t}\mb{E}_v\mc{W}_1(\mc{D}(x_t+\delta v),\mc{D}(x_{t-2}+\delta v)))+\georate^{\epoch_t}\cdot\gamma\|x_t-x_{t-1}\|\notag\\
&\leq \georate^{\epoch_t}\cdot \georate^2\mb{E}_vW_1(\dist_{t-3},\mc{D}(x_t+\delta v))+\georate^{\epoch_t}\cdot\gamma\mb{E}_v(\|x_t-x_{t-1}\|+\georate\cdot\|x_t-x_{t-2}\|)\notag\\
&\leq \georate^{\epoch_{t}}\georate^{t-1}\mb{E}_v{W}_1(\mc{D}(x_t+\delta v),\dist_0)+\georate^{\epoch_{t}}\gamma\sum_{i=1}^{t-1}\georate^{(i-1)}\mb{E}_v\|x_t-x_{t-i}\|.\label{eq:wassineq}
\end{align}
Hence, we need a bound on $\|x_t-x_{t-i}\|$ for each $i\in\{1,\ldots, t-1\}$.
Using the fact that $x_t=x_{t-1}-\eta_t\frac{\dimm}{\delta}\ell(x_{t-1}+\delta v_{t-1},z_{t-1})v_{t-1}$ where $z_{t-1}\sim \dist_{t-1}$, we have that
\begin{align*}
\|x_t-x_{t-i}\|
&=\|x_{t-1}-\eta_{t-1}\tfrac{\dimm}{\delta}\ell(x_{t-1}+\delta v_{t-1},z)v_{t-1}-x_{t-i}\|
\\&=\|x_{t-2}-\eta_{t-2}\frac{\dimm}{\delta}\ell(x_{t-2}+\delta v_{t-2},z_{t-2})v_{t-2}-\eta_{t-1}\frac{\dimm}{\delta}\ell(x_{t-1}+\delta v_{t-1},z_{t-1})v_{t-1}-x_{t-i}\|
\\&\leq\frac{\dimm}{\delta}\eta_1\sum_{j=t-i}^{t-1}|\ell(x_{j}+\delta v_{j},z_j)|\|v_j\|
\\&\leq\eta_1\frac{\dimm}{\delta}\ell_\ast (i-1),
\end{align*}
where the penultimate inequality holds from the fact that the learning rate is non-increasing.

Hence, we have that
\begin{align*}
\mb{E}_v{W}_1(\dist_{t},\mc{D}(x_t+\delta v))
&\leq \georate^{\epoch_{t}}\georate^{t-1}\mb{E}_v{W}_1(\mc{D}(x_t+\delta v),\dist_0)+\georate^{\epoch_{t}}\gamma\sum_{i=1}^{t-1}\georate^{(i-1)}\|x_t-x_{t-i}\|\\
&\leq \georate^{\epoch_{t}}\georate^{t-1}\mb{E}_v{W}_1(\mc{D}(x_t+\delta v),\dist_0)+\georate^{\epoch_{t}}\gamma\sum_{i=1}^{t-1}\georate^{(i-1)}\eta_1\frac{\dimm}{\delta}\ell_\ast (i-1)\\
&\leq \georate^{\epoch_{t}}\overline{W}(\dist_0)+\georate^{\epoch_{t}}\frac{4\gamma\dimm}{\alpha\delta}\frac{\ell_\ast\georate}{(1-\lambda)^2},
\end{align*}
where the last inequality holds using the fact that $\sum_{i=1}^{t-1}\georate^{(i-1)}(i-1)\leq \frac{\georate}{(1-\georate)^2}$.
\paragraph{Bounding gradient error.} Using this bound, we deduce 
\begin{align*}
    \|\nabla \mb{E}_{v\sim\mb{B}}[\mb{E}_{z\sim \dist_t}[\ell(z,x_t+\delta v)]]-\nabla {\LL}^\delta(x_t)\|^2
    &\leq \mb{E}_{v\sim\mb{B}}[L^2({W}_1(\dist_t,\mc{D}(x_t+\delta v)))^2]\\
    &\leq L^2\left(\georate^{\epoch_{t}}\overline{W}(\dist_0)+\georate^{\epoch_{t}}\frac{4\gamma\dimm}{\alpha\delta}\frac{\georate\ell_\ast}{(1-\lambda)^2}\right)^2.
\end{align*}
This concludes the proof.
\end{proof}

\begin{lemma}
Suppose that Assumptions~\ref{a:standing} and \ref{a:smooth_ell} hold. The loss $\LL^\delta(x)$ is differentiable and the map $x\mapsto \nabla \LL^\delta(x)$ is $\LLlip$-Lipschitz continuous. Moreover, the estimate holds:
\[\|\nabla \LL(x)-\nabla \LL^\delta(x)\|\leq \LLlip \delta\ \ \ \ \forall \ x\in \X.\]
\label{lem:smoothness_perturbed_problem}
\end{lemma}
\begin{proof}
For any point $x,x'\in \X$, we successively estimate
\[\|\nabla \LL^\delta(x)-\nabla \LL^\delta(x')\|\leq \E_{w\sim \mb{B}}[\|\nabla \LL(x+\delta w)-\nabla \LL(x'+\delta w)\|]\leq \LLlip \|x-x'\|\]
Thus $\nabla \LL^\delta$ is $\LLlip$-Lipschitz continuous. Next, we estimate
\[\|\nabla \LL(x)-\nabla \LL^\delta(x)\|\leq \E_{w\sim \mb{B}}[\|\nabla \LL(x+\delta w)-\nabla \LL(x)\|]\leq \LLlip\cdot\delta\E_{w\sim \mb{B}}\|w\|\leq \LLlip\cdot\delta,\]
which concludes the proof.
\end{proof}
Define the smoothed loss at $\dist_t$ as
\[\LL_t^\delta(x):=\E_{v\sim \mb{B}}[\E_{z\sim \dist_t}[\ell(z,x+\delta v)].\]
Let  $\bar{x}^{\delta}$ the optimal point of  $\LL^\delta$ on $(1-\delta)\X$, and $x^\delta$ be the optimal point of $\LL^\delta$ on $\X$. We have the following bound on the distance between the optimum of the performative prediction problem defined by $L$ on $\X$ and the optimum of the perturbed problem defined by $\LL^\delta$ on $(1-\delta)\X$.

The normal cone to a convex set $\X$ at $x\in \X$, denoted by $N_\X(x)$ is the set
\[N_{\X}(x)=\{v\in \mb{R}^d:\ \la v,y-x\ra\leq 0\ \ \forall y\in \X\}.\]

\begin{lemma}
Choose $\delta<\min\{r,\frac{\alpha}{H}\}$. Then the estimate holds:
\[\|x^\ast-\bar{x}^\delta\|\leq \delta \left(\left(1+\frac{\LLlip}{\alpha}\right)\|x^\ast\|+\frac{\LLlip}{\alpha}\right).\]
\label{lem:distance_between_opt}
\end{lemma}
\begin{proof}
There are two sources of perturbation: one replacing $\X$ with $(1-\delta)\X$ and the other in replacing $\LL$ with $\LL^\delta$. We will deal with each one individually. To do so, set $\tau:=1-\delta$ and let $\tilde{x}$ be the optimal point for $\LL$ on the shrunken set $\tau\X$. Thus $\tilde{x}$ satisfies the inclusion $0\in \nabla \LL(\tilde{x})+N_{\tau\X}(\tilde{x})$ where $N_{\tau\X}(\tilde{x})$ denotes the normal cone to $\tau\X$ at $\tilde{x}$. The triangle inequality directly gives
\begin{equation}
    \|x^\ast-\bar{x}^\delta\|\leq \|x^\ast-\tilde{x}\|+\|\tilde{x}-\bar{x}^\delta\|.
    \label{eq:split}
\end{equation}
Let us bound the first term on the right hand side of \eqref{eq:split}. To this end, since the map $x\mapsto \nabla \LL(x)+N_{\tau\X}(x)$ is $\alpha$-strongly monotone, we deduce 
\begin{equation}
\alpha\|\tilde{x}-\tau x^\ast\|\leq \text{dist}(0,\nabla \LL(\tau x^\ast)+N_{\tau\X}(\tau x^\ast)).
    \label{eq:dist_normal}
\end{equation}
Let use estimate the right hand side of \eqref{eq:dist_normal}. Since $x^\ast$ is optimal, the inclusion $0\in \nabla \LL(x^\ast)+N_{\X}(x^\ast)$ holds. Taking into account the identity $N_{\tau\X}(\tau x^\ast)=N_{\X}(x^\ast)$, we deduce
\[\text{dist}(0,\nabla \LL(\tau x^\ast)+N_{\tau \X}(\tau x^\ast))=\text{dist}(0,\nabla \LL(\tau x^\ast)+N_{\X}(x^\ast))\leq \|\nabla \LL(\tau x^\ast)-\nabla \LL(x^\ast)\|\leq \delta\cdot \LLlip\cdot\|x^\ast\|,\]
where the last inequality holds since $\nabla \LL$ is $\LLlip$-Lipschitz continuous. Appealing to \eqref{eq:dist_normal} and using the triangle inequality, we therefore deduce
\[\|x^\ast-\tilde{x}\|\leq \|\tilde{x}-\tau x^\ast\|+\delta\|x^\ast\|\leq \delta\left(1+\frac{\LLlip}{\alpha}\right)\|x^\ast\|.\]
It remains to upper bound $\|\tilde{x}-x^\ast\|$. Since $\tilde{x}$ is optimal, we have that
\begin{equation}
  \la -\nabla \LL(\tilde{x}),x-\tilde{x}\ra\leq 0, \ \ \forall \ x\in \tau \X.
  \label{eq:opt_tilde_theta}
\end{equation}
Analogously, since $\bar{x}^\delta$ is also optimal, we have that
\begin{equation}
  \la -\nabla \LL^\delta(\bar{x}^\delta),x-\bar{x}^\delta\ra\leq 0, \ \ \forall \ x\in \tau \X.
  \label{eq:opt_bar_theta}
\end{equation}
Then, by strong convexity and estimates \eqref{eq:opt_tilde_theta} and \eqref{eq:opt_bar_theta}, we get that
\begin{align*}
    \alpha\|\tilde{x}-\bar{x}^\delta\|^2&\leq \la \nabla \LL(\tilde{x})-\nabla \LL(\bar{x}^\delta),\tilde{x}-\bar{x}^\delta\ra\\
    &\leq \la \nabla \LL^\delta(\bar{x}^\delta)-\nabla \LL(\bar{x}^\delta),\tilde{x}-\bar{x}^\delta\ra\\
    &\leq \|\nabla \LL^\delta(\bar{x}^\delta)-\nabla \LL(\bar{x}^\delta)\|\|\tilde{x}-\bar{x}^\delta\|\\
    &\leq \LLlip\cdot \delta\cdot \|\tilde{x}-\bar{x}^\delta\|
\end{align*}
where the last inequality follows from Lemma~\ref{lem:smoothness_perturbed_problem}.
\end{proof}

The following lemma holds by a simple inductive argument.
\begin{lemma}\label{lem:technicalconverge}
Consider a sequence $D_t\geq 0$ for $t\geq 1$ and constants $t_0\geq 0$, $a>0$ satisfying
\[D_{t+1}\leq \left(1-\frac{2}{t+t_0}\right)D_t+\frac{a}{(t+t_0)^2}.\]
Then the estimate holds:
\[D_t\leq \frac{\max\{(1+t_0)D_1,a\}}{t+t_0}\quad \forall t\geq 1.\]
\end{lemma}
\subsection{Proof of Theorem~\ref{thm:zo_convergence_empirical}}
\label{app:newzeroorder}

\theoremzeroordernew*
\begin{proof}
\renewcommand{\LL}{\mc{L}}
Adding and subtracting appropriately, we have that
\begin{align*}
    \frac{1}{2}\|x_{t+1}-x^\ast\|^2&\leq \|x_{t+1}-\bdelta\|^2+\|\bdelta-x^\ast\|^2\\
    &\leq\|x_{t+1}-\bdelta\|^2+\delta^2\left(\left(1+\frac{\LLlip}{\alpha}\right)\|x^\ast\|+\frac{\LLlip}{\alpha}\right)^2.
\end{align*}
Now, to bound $\|x_{t+1}-\bdelta\|$, we have that
\begin{align*}
    \E[\|x_{t+1}-\bdelta\|^2]&\leq \E[\|x_{t}-\bdelta-\eta_t\hat{g}_t(x_t)\|^2]\\
    &\leq \E[\|x_{t}-\bdelta\|^2]-2\eta_t\E\la \hat{g}_t(x_t),x_{t}-\bdelta\ra+\eta_t^2\E\|\hat{g}_t(x_t)\|^2\\
    &=\E[\|x_{t}-\bdelta\|^2]-2\eta_t\E\la \nabla \LL_t^\delta(x_t),x_{t}-\bdelta\ra+\eta_t^2\E\|\hat{g}_t(x_t)\|^2
\end{align*}
 where the last equality holds since $\E[\hat{g}_t(x_t)]=\nabla \LL_t^\delta(x_t)$. We rewrite the smoothed gradient of the loss at time $t$ as 
\[\nabla \LL_t^\delta(x_t)=\nabla \LL^\delta(x_t)+\nabla \LL_t^\delta(x_t)-\nabla \LL^\delta(x_t).\]
Hence
\begin{align*}
    \E[\|x_{t+1}-\bdelta\|^2]&\leq \E[\|x_{t}-\bdelta-\eta_t\hat{g}_t(x_t)\|^2]\\
    &\leq \E[\|x_{t}-\bdelta\|^2]-2\eta_t\E\la \nabla \LL^\delta(x_t),x_{t}-\bdelta\ra-2\eta_t\E\la  \nabla\LL_t^\delta(x_t)-\nabla \LL^\delta(x_t),x_{t}-\bdelta\ra+\eta_t^2\E\|\hat{g}_t(x_t)\|^2\\
    &\leq \left(1-\eta_t\alpha\right)\E[\|x_{t}-\bdelta\|^2]-2\eta_t\E\la  \nabla\LL_t^\delta(x_t)-\nabla \LL^\delta(x_t),x_{t}-\bdelta\ra+\eta_t^2\frac{\ell_\ast^2d^2}{2\delta^2},
\end{align*}
where we used the fact that the smoothed loss is $(1-c)\alpha$ strongly convex for any $c\in (0,1)$ and we let $c:=1/2$.
Using the fact that 
\begin{align*}
    \E|\la \nabla \LL_t^\delta(x_t)-\nabla \LL^\delta(x_t),x_{t}-\bar{x}^\delta\ra|&\leq \frac{1}{2\Delta_1}L^2\left(\georate^{\epoch_{t}}\overline{W}(\dist_0)+\georate^{\epoch_{t}}\frac{4\gamma\dimm}{\alpha\delta}\frac{\ell_\ast}{(1-\lambda)^2}\right)^2+\frac{\Delta_1 \E \|x_{t}-\bar{x}^\delta\|^2}{2},
\end{align*}
we have that
\begin{align*}
    \E[\|x_{t+1}-\bdelta\|^2]
    &\leq \left(1-\eta_t\alpha\right)\E[\|x_{t}-\bdelta\|^2]+\eta_t^2\frac{\ell_\ast^2d^2}{2\delta^2}\\
    &\quad +2\eta_t\left(\frac{1}{2\Delta_1}L^2\left(\georate^{\epoch_{t}}\overline{W}(\dist_0)+\georate^{\epoch_{t}}\frac{4\gamma\dimm}{\alpha\delta}\frac{\ell_\ast}{(1-\lambda)^2}\right)^2+\frac{\Delta_1 \E \|x_{t}-\bar{x}^\delta\|^2}{2}\right)\\
    &=\left(1-\eta_t(\alpha-\Delta_1)\right)\E[\|x_{t}-\bdelta\|^2]+\eta_t^2\frac{\ell_\ast^2d^2}{2\delta^2}+2\eta_t\left(\frac{1}{2\Delta_1}L^2\left(\georate^{\epoch_{t}}\overline{W}(\dist_0)+\georate^{\epoch_{t}}\frac{4\gamma\dimm}{\alpha\delta}\frac{\ell_\ast}{(1-\lambda)^2}\right)^2\right)\\
    &\leq (1-\eta_t\frac{\alpha}{2})\E[\|x_{t}-\bdelta\|^2]+\eta_t^2\frac{\ell_\ast^2d^2}{2\delta^2} +2\eta_t\left(\frac{L^2}{\alpha}\left(\georate^{\epoch_{t}}\overline{W}(\dist_0)+\georate^{\epoch_{t}}\frac{4\gamma\dimm}{\alpha\delta}\frac{\ell_\ast}{(1-\lambda)^2}\right)^2\right)
\end{align*}
where we use $\Delta_1:=\alpha/2$.
Now, 
since 
\[n_t\geq \log\left(\frac{\overline{W}(\dist_0)+\frac{4\gamma\dimm}{\alpha\delta}\frac{\ell_\ast}{(1-\lambda)^2}}{\left(\eta_t\frac{\alpha}{L^2}\frac{\ell_\ast^2d^2}{4\delta^2}\right)^{1/2}}\right)\frac{1}{\log(1/\georate)},\]
we have that
\[\georate^{n_t}\left(\overline{W}(\dist_0)+\frac{4\gamma\dimm}{\alpha\delta}\frac{\ell_\ast}{(1-\lambda)^2}\right)\leq \left(\eta_t\frac{\alpha}{L^2}\frac{\ell_\ast^2d^2}{4\delta^2}\right)^{1/2},\]
so that
\[2\eta_t\left(\frac{L^2}{\alpha}\left(\georate^{\epoch_{t}}\overline{W}(\dist_0)+\georate^{\epoch_{t}}\frac{4\gamma\dimm}{\alpha\delta}\frac{\ell_\ast}{(1-\lambda)^2}\right)^2\right)\leq \eta_t^2\frac{\ell_\ast^2d^2}{2\delta^2}.\]
Therefore, we deduce
\begin{align*}
    \E[\|x_{t+1}-\bdelta\|^2]
    &\leq \left(1-\eta_t\frac{\alpha}{2}\right)\E[\|x_{t}-\bdelta\|^2]+\eta_t^2\frac{\ell_\ast^2d^2}{\delta^2}.
\end{align*}
Since $\eta_t=4/(\alpha t)$, we apply Lemma~\ref{lem:technicalconverge} to deduce that
\[\E\|x_{t+1}-\bdelta\|^2\leq \frac{\max\{\alpha^2\delta^2\|x_1-\bdelta\|^2,16\ell_\ast^2 d^2\}}{\delta^2\alpha^2t}\quad \forall t\geq 1.\]
This concludes the proof.

\end{proof}

\subsection{Proof of Corollary~\ref{cor:zo_convergence}}
\corollaryzeroorderrate*
\begin{proof}
The assumed upper bound on $\varepsilon$ directly implies that $\delta\leq \frac{\alpha}{2\LLlip}$ and $\delta<r$. An application of Theorem~\ref{thm:zo_convergence_empirical} yields the estimate
\[\E[\|x_t-x^\ast\|^2]\leq \frac{\max\{\delta^2\alpha^2\|x_1-\bdelta\|^2,16d^2\ell_\ast^2\}}{t\alpha^2\delta^2}+\frac{\varepsilon}{2}.\]
Setting the right side to $\varepsilon$, solving for $t$, and using the trivial upper bound $\|x_1-\bdelta\|\leq 2R$ completes the proof. 
\end{proof}

\section{Proofs for First Order Oracle Setting}
\label{sec:missing_proofs}

\subsection{Proof of Lemma \ref{lemma:bounded_gradient}}
\label{sec:proof_lemma_bounded_gradient}
\lemmaboundedgradient*
\begin{proof}
Observe that 
using Jensen's inequality along with Lemma~\ref{lemma:kr_duality2},  we deduce
\begin{align*}
    \|\nabla \mb{E}_{z\sim \dist_t}[\ell(z,x_t)]]-\nabla {\LL}(x_t)\|^2&= \left[\|\nabla \mb{E}_{z\sim \dist_t}[\ell(z,x_t)]-\nabla \mb{E}_{z\sim \mc{D}(x_t)}[\ell(z,x_t)]\|^2 \right]\\
    &\leq L^2({W}_1(\dist_t,\mc{D}(x_t)))^2.
\end{align*}
The remainder of the proof is identical to the proof of Lemma~\ref{lemma:expected_bounded_gradient}. Indeed, we have
that
\begin{align*}
{W}_1(\dist_{t},\mc{D}(x_t))
&\leq \georate^{\epoch}\georate^{t-1}{W}_1(\mc{D}(x_t),\dist_0)+\georate^{\epoch}\gamma\sum_{i=1}^{t-1}\georate^{(i-1)}\|x_t-x_{t-i}\|.
\end{align*}
Hence, we need a bound on $\|x_t-x_{t-i}\|$ for each $i\in\{1,\ldots, t-1\}$.
Recall that $x_t=x_{t-1}-\eta_{t-1}\hat{g}_{t-1}$ where
\[\hat{g}_{t-1}=\nabla_x\ell(x_{t-1},z_{t-1})+(1-\georate^{n})A^\top \nabla_z\ell(x_{t-1},z_{t-1}), \quad\text{and}\quad z_{t-1}\sim \dist_{t-1}.\]
Moreover, 
\[\|\hat{g}_t\|\leq L(1+\|A\|_{\op})\]
since $\ell$ is $L$-Lipschitz continuous.
Hence, we have the following bound:
\begin{align*}
\|x_t-x_{t-i}\|
&=\|x_{t-1}-\eta_{t-1}\hat{g}_{t-1}-x_{t-i}\|
\\&=\|x_{t-2}-\eta_{t-2}\hat{g}_{t-2}-\eta_{t-1}\hat{g}_{t-1}-x_{t-i}\|
\\&\leq\eta_1\sum_{j=t-i}^{t-1}\|\hat{g}_j\|
\\&\leq\eta_1\cdot L\cdot (1+\|A\|_{\op})\cdot (i-1),
\end{align*}
where the penultimate inequality holds from the fact that the learning rate is non-increasing.

Therefore, we deduce
\begin{align*}
{W}_1(\dist_{t},\mc{D}(x_t))
&\leq \georate^{\epoch}\georate^{t-1}{W}_1(\mc{D}(x_t),\dist_0)+\georate^{\epoch_{t}}\gamma\sum_{i=1}^{t-1}\georate^{(i-1)}\eta_1\cdot L\cdot (1+\|A\|_{\op})\cdot (i-1)\\
&\leq \georate^{\epoch}\georate^{t-1}{W}_1(\mc{D}(x_t),\dist_0)+\georate^{\epoch}\gamma\eta_1\cdot L\cdot (1+\|A\|_{\op})\cdot\frac{\georate}{(1-\georate)^2},
\end{align*}
where the last inequality follows from the fact that $\sum_{i=1}^{t-1}\georate^{(i-1)}(i-1)\leq \frac{\lambda}{(1-\lambda)^2}$. Using this bound on the Wasserstein-1 distance between the current probability distribution $\dist_t$ at time $t$ and the fixed point probability distribution $\mc{D}(x_t)$ induced by $x_t$, we have that
\begin{align*}
    \|\nabla \mb{E}_{z\sim \dist_t}[\ell(z,x_t)]]-\nabla {\LL}(x_t)\|^2
    \leq L^2\cdot\left(\georate^{\epoch}\overline{W}_1(\dist_0)+\georate^{\epoch}\gamma\eta_1\cdot L\cdot (1+\|A\|_{\op})\cdot\frac{\georate}{(1-\georate)^2}\right)^2
\end{align*}
since $\lambda^{t-1}\leq 1$.
This concludes the proof.
\end{proof}

\newcommand{\G}{\mc{G}}
\subsection{Proof of Theorem~\ref{thm:fo_convergence}}\label{app:fo_convergence_empirical}
We restate the theorem for convenience.
\theoremfirstorder*
Note that the gradient $\hat{g}_t$ approximates the gradient $\G(x):=\nabla \LL(x)=\nabla \E_{z\sim \mc{D}(x)}\ell(x,z)$.
\begin{align*}
    \|x_{t+1}-x_t\|&=\|x_t-\eta\hat{g}_t-x_t\|=\|x_t-\eta\G(x_t)-\eta(\hat{g}_t-\G(x_t))-x_t\|.
\end{align*}
Noting that $x_{t+1}$ is the minimizer of the $1$--strongly convex function $x\mapsto \frac{1}{2}\|x_t-\eta\hat{g}_t-x\|^2$ over $\X$, we deduce 
\[\frac{1}{2}\|x_{t+1}-x^\ast\|^2\leq \frac{1}{2}\|x_t-\eta\hat{g}_t-x^\ast\|^2-\frac{1}{2}\|x_t-\eta\hat{g}_t-x_{t+1}\|^2.\]
Expanding the squares on the right hand side and combining terms yields
\begin{align*}
    \frac{1}{2}\|x_{t+1}-x^\ast\|^2&\leq \frac{1}{2}\|x_t-x^\ast\|^2-\eta\la \hat{g}_t,x_{t+1}-x^\ast\ra-\frac{1}{2}\|x_{t+1}-x_t\|^2\\
    &=\frac{1}{2}\|x_t-x^\ast\|^2-\eta\la \hat{g}_t,x_t-x^\ast\ra-\frac{1}{2}\|x_{t+1}-x^\ast\|^2-\eta\la \hat{g}_t,x_{t+1}-x_t\ra.
\end{align*}
Setting $\mu_t:=\E_t[\hat{g}_t]$, we successively compute
\begin{align*}
    \frac{1}{2}\mb{E}_t\|x_{t+1}-x^\ast\|^2&\leq \frac{1}{2}\|x_t-x^\ast\|^2-\eta\la \mb{E}_t\hat{g}_t,x_t-x^\ast\ra-\frac{1}{2}\mb{E}_t\|x_{t+1}-x^\ast\|^2-\eta\mb{E}_t\la \hat{g}_t,x_{t+1}-x_t\ra\\
    &\leq \frac{1}{2}\|x_t-x^\ast\|^2-\eta\la \mu_t,x_t-x^\ast\ra-\frac{1}{2}\mb{E}_t\|x_{t+1}-x^\ast\|^2-\eta\mb{E}_t\la \hat{g}_t,x_{t+1}-x_t\ra\\
    &=\frac{1}{2}\|x_t-x^\ast\|^2-\eta\mb{E}_t\la \G(x_{t+1}),x_{t+1}-x^\ast\ra-\frac{1}{2}\mb{E}_t\|x_{t+1}-x^\ast\|^2\\
    &\quad+\eta \underbrace{\mb{E}_t\la \hat{g}_t-\mu_t,x_t-x_{t+1}\ra}_{P_1}+\eta\underbrace{\mb{E}_t[\la \mu_t-\G(x_{t+1}),x^\ast-x_{t+1}\ra]}_{P_2}.
\end{align*}
Strong convexity of $\LL(x)$ implies that $\la \G(x_{t+1}),x_{t+1}-x^\ast\ra\geq \alpha\|x_{t+1}-x^\ast\|^2$ so that
\[\frac{1+2\eta\alpha}{2}\mb{E}_t\|x_{t+1}-x^\ast\|^2\leq \frac{1}{2}\|x_t-x^\ast\|^2-\frac{1}{2}\mb{E}_t\|x_{t+1}-x_t\|^2+\eta(P_1+P_2).\]
Using Young's inequality, we upper bound $P_1$ as follows:
\begin{align*}
    P_1&\leq \frac{1}{2\Delta_1}\mb{E}_t\|\hat{g}_t-\mu_t\|^2+\frac{\Delta_1\mb{E}_t\|x_{t+1}-x_t\|^2}{2}\\
    &\leq \frac{\sigma^2}{2\Delta_1}+\frac{\Delta_1\mb{E}_t\|x_{t+1}-x_t\|^2}{2}
\end{align*}
using Assumption~\ref{a:finite_variance}.
Using Yong's inequality again, we have that
\begin{align*}
    P_2\leq \frac{\mb{E}_t\|\mu_t-\G(x_{t+1})\|^2}{2\Delta_2}+\frac{\Delta_2\mb{E}_t\|x_{t+1}-x^\ast\|^2}{2}.
\end{align*}
Next observe that
\begin{align*}
    \mb{E}_t\|\mu_t-\G(x_{t+1})\|^2&\leq 2\mb{E}_t\|\mu_t-\G(x_t)\|^2+2\mb{E}_t\|\G(x_t)-\G(x_{t+1})\|^2\\
    &\leq 2C^2+2\LLlip^2\mb{E}_t\|x_t-x_{t+1}\|^2,
\end{align*}
where
\[C^2:= L^2\cdot\left(\georate^{\epoch}\overline{W}_1(\dist_0)+\georate^{\epoch}\gamma\eta\cdot L\cdot (1+\|A\|_{\op})\cdot\frac{\georate}{(1-\georate)^2}\right)^2.\]
Therefore
\begin{equation}\label{eq:P2bdd}
P_2\leq \frac{2C^2+2G^2\|x_t-x_{t+1}\|^2}{2\Delta_2}+\frac{\Delta_2\mb{E}_t\|x_{t+1}-x^\ast\|^2}{2}.    
\end{equation}
Now we have that
\begin{equation}\label{eq:bdd_fo_a}
    \begin{aligned}
    \frac{1+\eta(2\alpha-\Delta_2)}{2}\mb{E}_t\|x_{t+1}-x^\ast\|^2&\leq \frac{1}{2}\|x_{t}-x^\ast\|^2+\frac{\eta\sigma^2}{2\Delta_1}+\frac{\eta C^2}{\Delta_2}-\frac{1-2\eta\LLlip^2\Delta_2^{-1}-\eta\Delta_1}{2}\mb{E}_t\|x_{t+1}-x_t\|^2.
    \end{aligned}
\end{equation}
Setting $\Delta_2=\alpha$ and $\Delta_1=\frac{1}{\eta}-\frac{2\LLlip^2}{\alpha}$ ensures the last term on the right hand side is zero. We also have that $\eta\leq \alpha/(4\LLlip^2)$ implies that $\Delta_1\geq \frac{1}{2\eta}$. Rearranging \eqref{eq:bdd_fo_a} we get that
\[\mb{E}_t\|x_{t+1}-x^\ast\|^2\leq \frac{1}{1+\eta\alpha}\|x_t-x^\ast\|^2+\frac{2\eta^2\sigma^2}{1+\eta\alpha}+\frac{2\eta C^2}{\alpha(1+\eta\alpha)}.\]
Next we verify that our choice of $n$ is large enough so that 
$\frac{C^2}{\alpha}\leq \eta\sigma^2$.
Indeed, this is equivalent to 
\[\left(\georate^{\epoch}\overline{W}_1(\dist_0)+\georate^{\epoch}\gamma\frac{4}{\alpha G^2}\cdot L\cdot (1+\|A\|_{\op})\cdot\frac{\georate}{(1-\georate)^2}\right)\leq \frac{\alpha^{1/2}}{L}\eta^{1/2}\sigma\]
which is in turn equivalent to 
\[n\geq \log\left(L\frac{\overline{W}_1(p_0)+\gamma\frac{4}{\alpha G^2} L(1+\|A\|_{\op})\frac{\georate}{(1-\georate)^2}}{(\alpha\eta)^{1/2}\sigma}\right)\frac{1}{\log(1/\georate)}.\]
Hence, for our choice of $n$, we have that
\[\mb{E}_t\|x_{t+1}-x^\ast\|^2\leq \frac{1}{1+\eta\alpha}\|x_t-x^\ast\|^2+\frac{4\eta^2\sigma^2}{1+\eta\alpha}.\]
Which completes the proof.
\section{Numerical Simulations}
\label{app:sfpark}
In this section, we start by describing the SFpark data and experiment set-up. Then we provide additional figures and details for each of the two experiments conducted in the main. Finally, we introduce a synthetic data example which abstracts strategic classification in settings where agents have memory.

\subsection{SFPark Data Description}
\label{app:sfpark-desc}
In this section, we provide more details on our data cleaning strategies and our model for the SFpark dataset. 

\paragraph{Data cleaning.} We start by discussing our data cleaning strategy. Of the many features in the dataset, the key ones of interest to us were the street name, district name, total time available (number of parking spots multiplied by number of seconds per hour), total time occupied, and rate. Many of the rates were unavailable in the original dataset, but the rate charged for the day before and day after were. If we encountered a missing rate, we replaced it with the rate before and after, if those rates were equal. We only worked with blocks where we could successfully fill in each of the missing rates. This process can be found in the accompanying code.

\paragraph{Estimating price sensitivity.} 
The model we consider is explained in the main body. To provide more intuition and details,  as an example, consider the $600$ block of Beach Street (Beach ST $600$) for the time window between $1200$--$1500$. The initial distribution, $d_0$, is sampled from the data at the initial price for parking along Beach ST $600$, which in this case is $x_0=\$3$ per hour. As described in Section \ref{sec:experiments}, we assume that for an announced price difference of $x=\tilde{x}-x_0$, $\tilde{x}$ is the charged price and $x$ is the variable of optimization. The occupancy follows a distribution of $\zeta + A(\tilde{x}-x_0)$, where $\zeta$ follows the same distribution as $p_0$. 

The \emph{price sensitivity} $A$ is a proxy for the price elasticity, in that it provides us a relationship between the change in price and the change in occupancy mapped to a $(0,1)$ scale. Indeed, recall that price elasticity is a change in the percentage occupancy for a given change in percentage price. Hence, price sensitivity as we have defined it has the same sign as price elasticity except that it is in the right units of our mathematical abstraction for the problem, and is in this sense a proxy thereof.  We compute $A$ by considering the following:
\begin{enumerate}[label={\alph*.}, itemsep=0pt]
    \item The average occupancy for the initial price over every weekday in the beginning of the pilot study until the price is changed.
    \item The average occupancy over every weekday in the final week of the last price announcement.
\end{enumerate}
 As an example, for  the $600$ block of Beach ST, the initial price was $\$3.00$ per hour and the average occupancy before a new price was announced was approximately $60.6$\%, the final price announced during the pilot study was $\$4.25$, and the average occupancy for the final week was approximately $41.1$\%. Therefore, for the $600$ block of Beach ST, we estimate that \[A\approx\frac{0.411-0.606}{4.25-3}=-0.156,\]
where occupancy percentage is mapped to the $[0,1]$ scale. 
It was shown in \citet{pierce2018sfpark} that price elasticity is in general a small negative number on average for the SFpark pilot study and experiment.  This is consistent with  prior studies on price elasticity for on-street parking where information about price and location plays a crucial role~\citep{fiez2020tits,glasnapp2014understanding}. However, for the SFpark pilot study, the price elasticity also depends highly on the block and neighborhood.
\begin{figure*}[t!]
    \centering
    \includegraphics[width=1.0\textwidth]{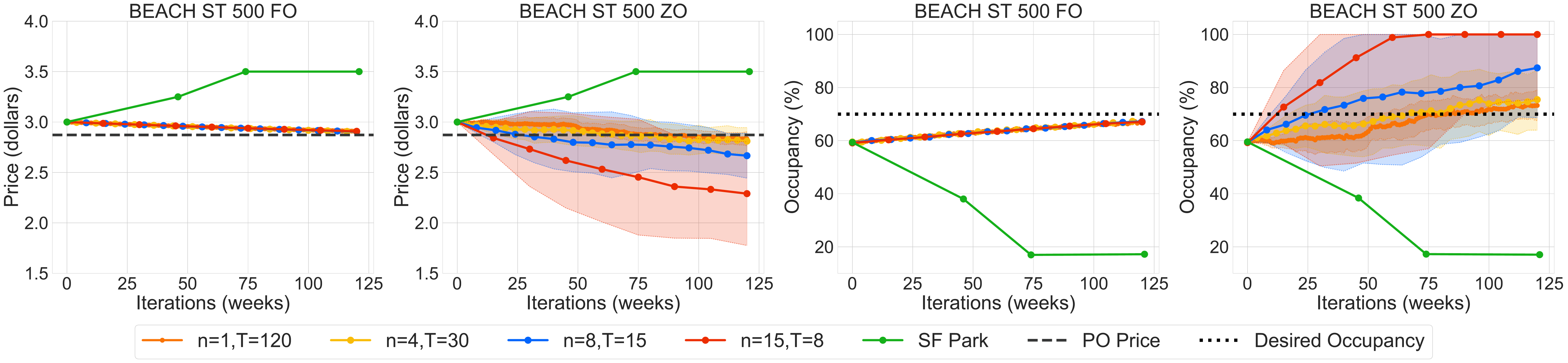}
    \includegraphics[width=1.0\textwidth]{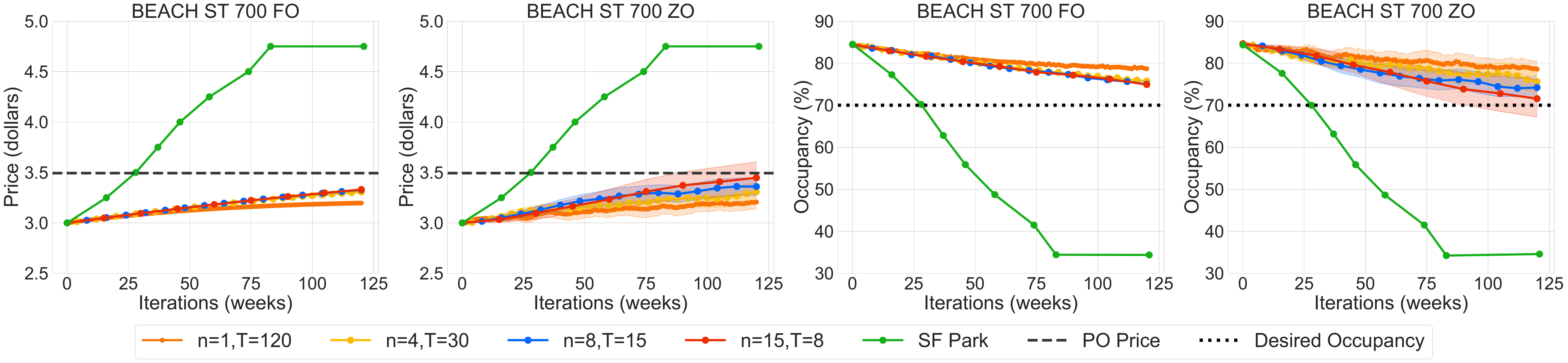}
    \includegraphics[width=1.0\textwidth]{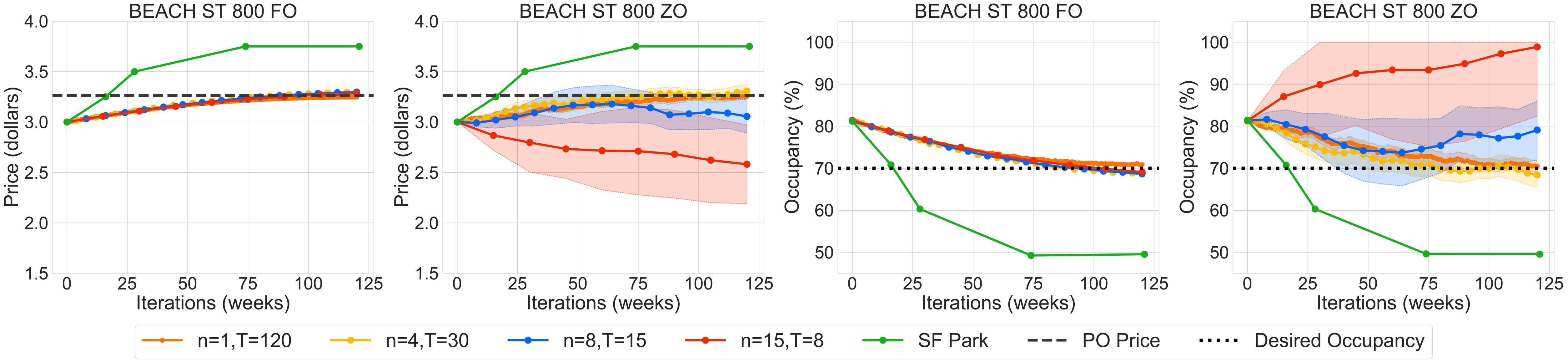}
    \caption{Results of Algorithm \ref{alg:main-fo} (first and third plots of each row) and  Algorithm \ref{alg:main-zo} (second and fourth plots of each row) with different $(n,T)$ pairs for the $500$, $700$ and $800$ blocks of Beach ST  for time window $1200$--$1500$.
    Each marker represents a price announcement, and the plots show the prices and corresponding predicted occupancies. The SFpark prices and occupancies are far from the target and performatively optimal price, whereas the proposed algorithms obtain both points up to theoretical error bounds.}
    \label{fig:beachst_fo_all_rest}
\end{figure*}
\begin{figure*}[t!]
    \centering
    \includegraphics[width=0.333\textwidth]{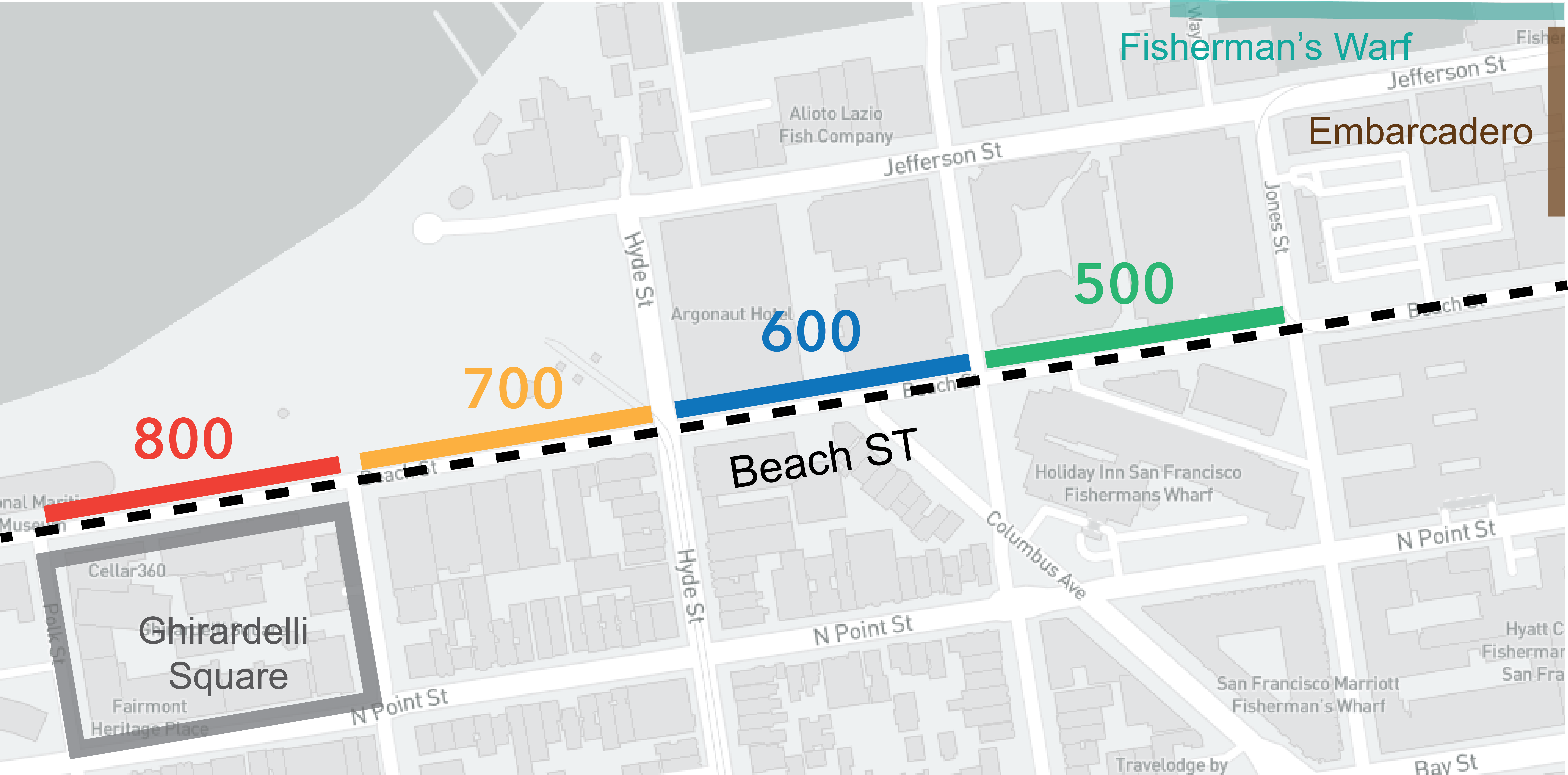}
    \caption{Map of Beach Street showing blocks $500$ to $800$. The tourist attractions Ghiradelli Square, Fisherman's Wharf, and the Embarcadero are also depicted. }
    \label{fig:beachst_map}
\end{figure*}
\paragraph{Estimating geometric decay parameter $\georate$.} We also use this data to estimate the geometric decay rate, $\georate$. As described in Section \ref{sec:experiments}, when a new rate is posted, the effect on the occupancy is not immediate, and so the geometric decay rate, $\georate$, in this context represents the speed at which this new announced price travels through the population (and consequently affects the parking occupancy). We group the occupancy data by day of week, in order to account for different traffic patterns on different weekdays. We assume that the week before a new price is announced is the fixed point distribution of the previous rate. For example, for  the $600$ block of Beach ST, a rate of $\$3.50$ per hour was announced on February $14$, $2012$, which means that we assumed that the occupancies on February $7$--$13$, $2012$ were the fixed point distributions of the previous rate $\$3.25$.
We now fix a day of the week (e.g., Monday), a block (e.g.,  Beach ST $600$), and a time window (e.g., $1200$--$1500$). Suppose the prices $\{x_i\}$ are announced and $\mc{D}(x_i)$ represents the fixed point distribution of announcing $x_i$, where the price $x_i$ is in effect for $K_i$ weeks. Then, for the $k$-th week after announcing $x_i$, we assume that the occupancy is represented by $\georate^k\mc{D}(x_{i-1})+(1-\georate^k)\mc{D}(x_i)$. For each week $k$, and for price $x_i$, the occupancy for the specified day is represented as $z_{i,k}$. To find the value of $\georate$, for the specified day and block, we solve the following optimization problem:
\begin{align*}
\underset{\georate\in [0,1]}{\text{minimize}} \quad & \sum_i\sum_{k=1}^{K_i}(\georate^k\mc{D}(x_{i-1})+(1-\georate^k)\mc{D}(x_i)-z_{i,k})^2.
\end{align*}
We perform projected gradient descent to solve this problem. For the final value of $\georate$ that we use for the specified block, we average the estimated values of delta for each day.

\subsection{Comparing Performative Optimum to SFpark}
\label{app:sfpark-compare-to-SFpark}

Here, we provide experiments for other blocks on Beach Street (beyond just the $600$ block in Section \ref{sec:experiments}). Each row in Figure \ref{fig:beachst_fo_all_rest} shows prices and corresponding occupancies for Algorithm \ref{alg:main-fo} and Algorithm \ref{alg:main-zo} for  the $500$, $700$, and $800$ blocks of Beach ST, respectively. In each instance, we make similar observations to those in Section \ref{sec:experiments} for the $600$ block on Beach ST, namely, that SFpark consistently overshot the price to reach the target occupancy, and that the choice of $n=8$ is reasonable, in that a time period of $8$ weeks is sufficient for the population to equilibriate before announcing a new price.

An interesting observation from Figure \ref{fig:beachst_fo_all_rest} comes from the fact that the $500$ block of Beach ST has a price sensitivity of $A\approx-0.844$, and the $800$ block of Beach ST has a price sensitivity of $A\approx-0.424$. Since both of these values have large magnitudes, we observe that for a small price reduction, the estimated occupancy increases to $100\%$. Therefore, for blocks where the magnitude of the price sensitivity is large, our experiments suggest using a smaller choice of $n$, and consequently a larger choice of $T$, in order to reduce the variance for the price announcements to prevent having large fluctuations in occupancy. All four of the blocks on Beach Street have very similar estimated $\georate$ values.
\begin{table}[h!]
    \centering
    \begin{tabular}{|c|c|}\hline
       Beach ST Block Number & (estimated) $\approx\georate$ value  \\\hline\hline
        $500$ & $0.993$\\
        $600$ & $0.959$\\
        $700$ & $0.993$\\
        $800$ & $0.984$\\\hline
    \end{tabular}
    \caption{Estimated decay rate $\georate$ for each block along Beach ST pictured in Figure  \ref{fig:beachst_map}.}
    \label{tab:deltas}
\end{table}
Table~\ref{tab:deltas} indicates that each block adjusts to new price announcements at similar rates. This makes sense given that each of the blocks are on the same street all next to each other as seen in Figure \ref{fig:beachst_map}, and located near similar landmarks and attractions.

\subsection{Redistributing Parking Demand}
\label{app:sfpark-redistribute}
In this appendix subsection, we describe the details for the experiment on redistributing parking demand. The study includes the four connected blocks of Hawthorne ST $0$, Hawthorne ST $100$, Folsom ST $500$, and Folsom ST $600$ because the blocks are adjacent to one another as shown in Figure \ref{fig:bar}. Thus, we wanted to investigate whether price changes would redistribute the traffic such that each block had an occupancy closer to the target of $70\%$. An interesting note is that while Folsom ST $500$ and Folsom ST $600$ both have negative price sensitivity values of of $A\approx-0.399$ and $A\approx-0.284$ respectively, Hawthorne ST $0$ and Hawthorne ST $100$ have positive price sensitivity values of $A\approx0.454$ and $A\approx0.044$ respectively. Since Hawthorne ST has a very high initial average occupancy, SFpark should consider decreasing prices on this street in order to shift demand to the nearby streets. This is exactly what we see done by both Algorithms  \ref{alg:main-zo} and \ref{alg:main-fo} so that both streets are closer to the target occupancy. Although the price sensitivity is very different for these blocks, the estimated $\georate$ values are very similar. Hawthorne ST $0$ has $\georate\approx0.853$, Hawthorne ST $100$ has $\georate\approx0.979$, Folsom ST $500$ has $\georate\approx0.996$, and Folsom ST $600$ has $\georate\approx0.793$, so each block adjusts to new price announcements at similar rates.

\subsection{Synthetic Data: Strategic Classification in Dynamic Environments}
\label{app:synthetic}

In this appendix subsection, we apply our algorithm to a synthetic strategic classification problem---which was considered in the dynamic setting in \citet{brown2020performative} and in the static setting in \citet{drusvyatskiy2020stochastic, miller2021outside, perdomo2020performative}, e.g.---where there is memory in the agent population. 
For simplicity (and to support visualization of the classifier performance), each data point contains a feature vector, $\phi_i\in\mb{R}^{2}$, and a corresponding label, $y_i\in\{-1,1\}$ where $i\in\{1,\ldots, m\}$ and $m$ is the number of strategic users.
The loss incurred by the institution is given by an $\ell_2$-regularized logistic loss:
\[\frac{1}{2}\sum_{i=1}^{m}-y_i\left<x,\phi_i\right>+\log(1+\exp(\left<x,\phi_i\right>))+\frac{\nu}{2}\|x\|^2,\]
where we set $m=1000$. The agents are non-strategic (meaning they do not perturb their true feature vector $\bar{\phi}_i$) if they have label $y_i=1$, and otherwise `best respond' to the announced classifier according to 
the model 
\begin{equation}
\phi_{i}=\arg\max_{w}-\langle x,w\rangle-\frac{1}{2\tilde{\epsilon}}\|w-\bar{\phi}_i\|^2=\bar{\phi}_i-\tilde{\epsilon}x.
\label{eq:best_response}
\end{equation}
We take $\tilde{\epsilon}=0.1$, but the observations we make hold more generally with the exception of very large magnitude perturbations for which the problem (even in the static setting) becomes untenable.
We randomly select a subset of the two features to treat as strategic. 
We also randomly generate a ground truth data set by drawing $m\times 2$ samples from a normal distribution, drawing the ground truth $\phi_{\text{gt}}$ from a (2 dimensional) normal distribution and then assigning labels according to 
\[y_i=(\text{sign}(\phi_i^\top \phi_{\text{gt}}+0.1v)+1)/2, \ v\sim \mc{N}(0,1).\]
Specifically, agents are allowed to perturb in the $x_1$ direction as can be seen in Figure~\ref{fig:seed6classifier}.
Moreover, we take the initial data distribution $p_0$ to be far from the base distribution for users' true preferences $\bar{\phi}_i$ even with performative effects; specifically, $p_0$ is a Gaussian distribution with a mean of $1.0$ and scale (standard deviation) of $45$. 
More details on the implementation can be found in the accompanying code. 

We divide the data into a training and test set with a $(2/3)$--$(1/3)$ split. 
We set the regularization parameter to $\nu = 1/m_{\text{train}}$ where $m_{\text{train}}$ is the size of the training data set. 
For \eqref{eq:best_response}, the inner product can be interpreted as the utility of the agent and the norm difference as the cost of manipulation. We present results for a modest value of $n=20$; similar or lower values are consistent with our observations and as our theory suggests, as $n\to\infty$, the solution obtained by Algorithm~\ref{alg:main-fo} approaches the performatively optimal solution. 
\begin{figure}[h!]
    \centering
    \includegraphics[width=0.75\textwidth]{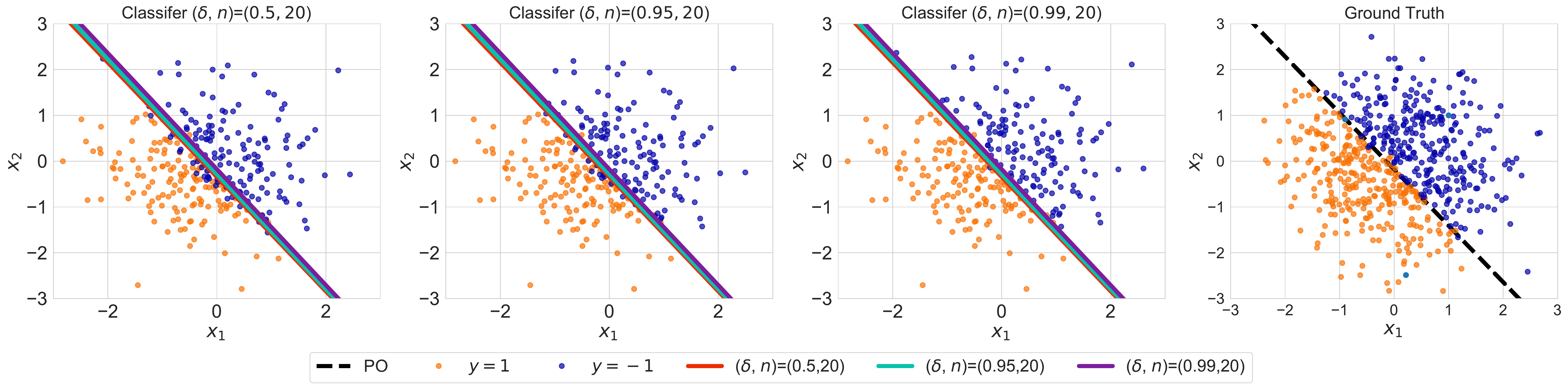}\includegraphics[width=0.25\textwidth]{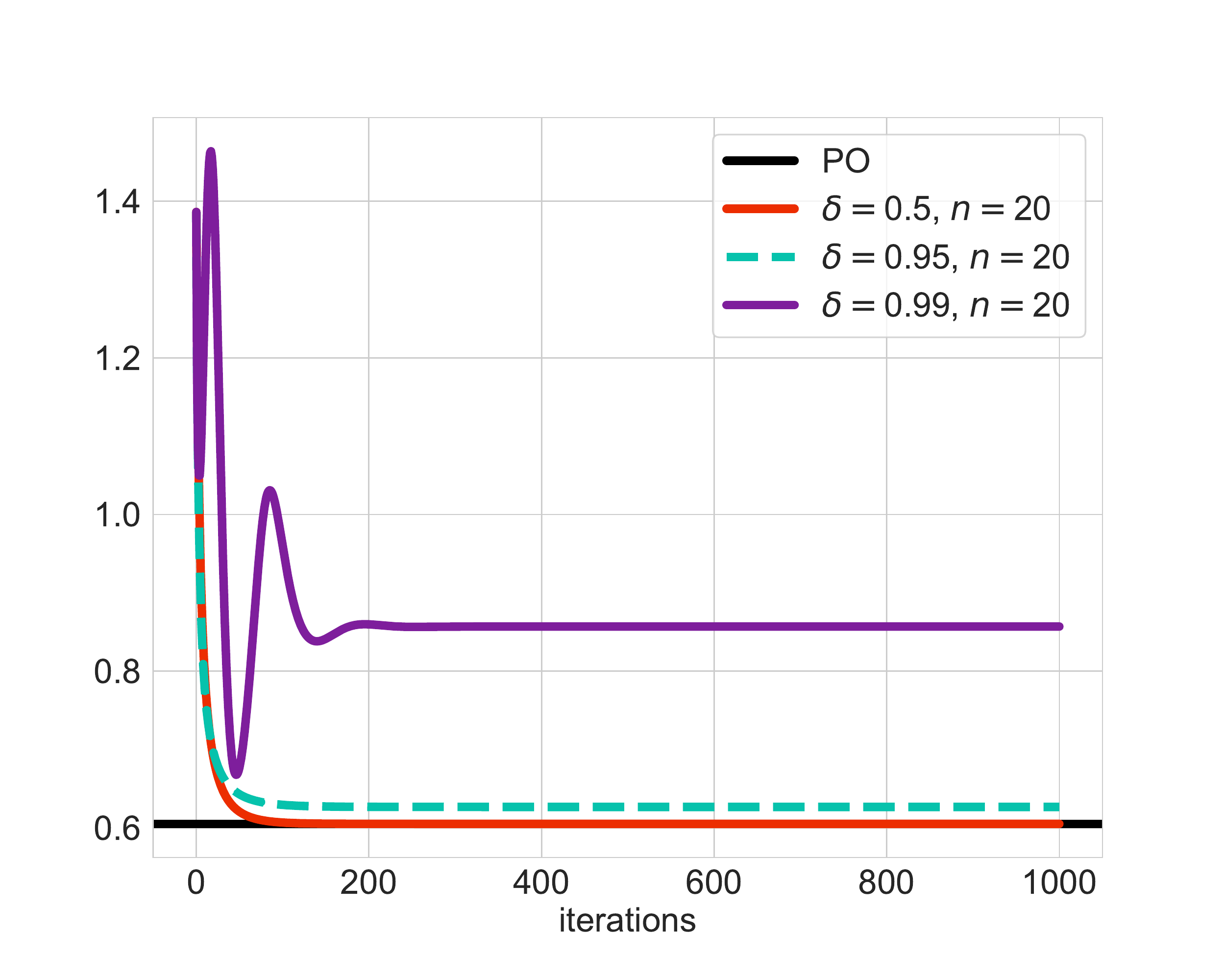}
    \caption{Classifiers and losses for different values of $\georate$ and $n=20$. In order of appearance from left to right, the first three plots show the learned classifiers with the data at the distribution $\mc{D}(x)$ induced by the learned classifier for $(\georate,n)=(0.5,20)$, $(\georate,n)=(0.95,20)$, $(\georate,n)=(0.99,20)$. The fourth plot from the left is the ground truth data distribution without performative effects. The differences in the data distributions are subtle, but one can see that the different learned classifiers evoke different responses from the strategic users. The far right plot shows the losses as a function of iterations.
    }
    \label{fig:seed6classifier}
\end{figure}
\begin{figure}[h!]
    \centering
    \includegraphics[width=0.85\textwidth]{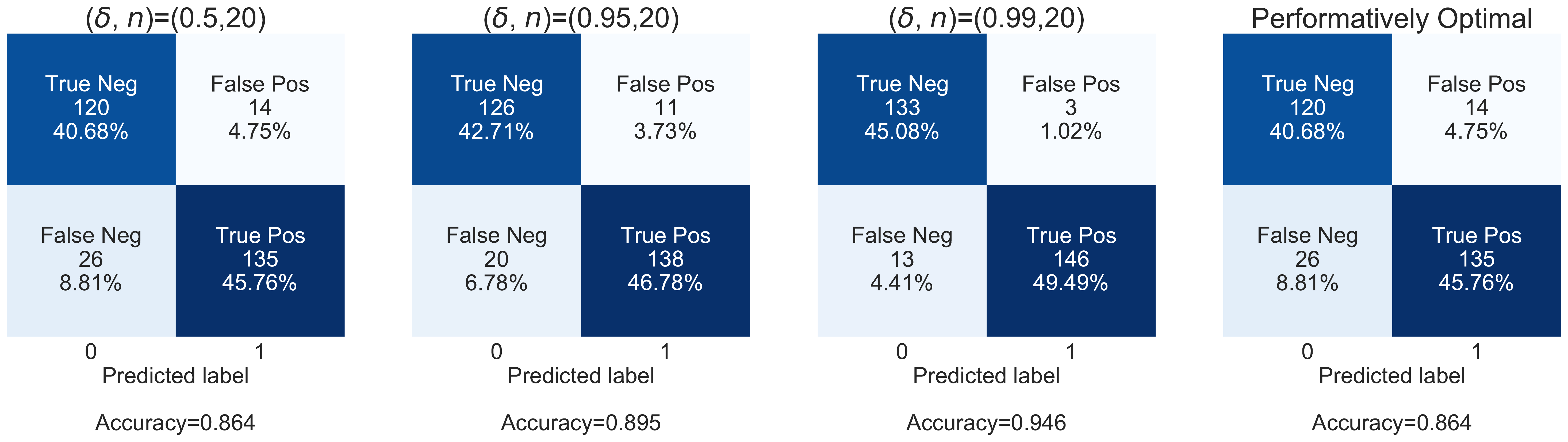}
    \caption{Accuracy of the classifiers (via confusion matrix) learned for the data distribution and setting shown in Figure \ref{fig:seed6classifier}. For this randomly sampled data distribution, $\georate$ plays a significant role on the generalization capability (as measured by accuracy on the test set). Surprisingly, accuracy improves as the mixing parameter $\georate$ increases (meaning longer time to mix) and this also has an impact on auxiliary but related metrics such as the false positive and false negative rates. This observation depends highly on the data distribution, but exposes interesting directions for future theoretical work on understanding how performative optimality translates to generalization and robustness guarantees.}
    \label{fig:seed6accuracy}
\end{figure}

We explore different values of $\georate$ and $n$---i.e., the mixing parameter of the geometric dynamics and the epoch length of Algorithm~\ref{alg:main-fo}---on not just convergence  but also on accuracy. The observations we report actually lead to a number of interesting open questions for this field including how performative optimality relates to generalization. We find that depending on the skew of the data distribution and the strength of the perturbation power of the strategic agents---namely, $\tilde{\epsilon}$---that surprisingly, the performatively optimal point may not generalize very well as compared to the solution obtained by Algorithm~\ref{alg:main-fo} when the mixing parameter $\georate$ is large. The latter has better accuracy as can be seen in Figure~\ref{fig:seed6accuracy}; the loss value per iteration and the classifiers for different $\georate$ values are shown in Figure~\ref{fig:seed6classifier}.

In other settings (e.g., with different ground truth data), the  solution obtained by Algorithm~\ref{alg:main-fo}, even with different values of $\georate$ and different choices of epoch length $n$, performs just as well as the performatively optimal solution as depicted in Figure~\ref{fig:seed980accuracy}, the data for which has original distribution depicted in Figure \ref{fig:seed980classifier}, which also contains the learned classifiers and losses per iteration for different $\georate$ values. 

\begin{figure}[h!]
    \centering
    \includegraphics[width=0.39\textwidth]{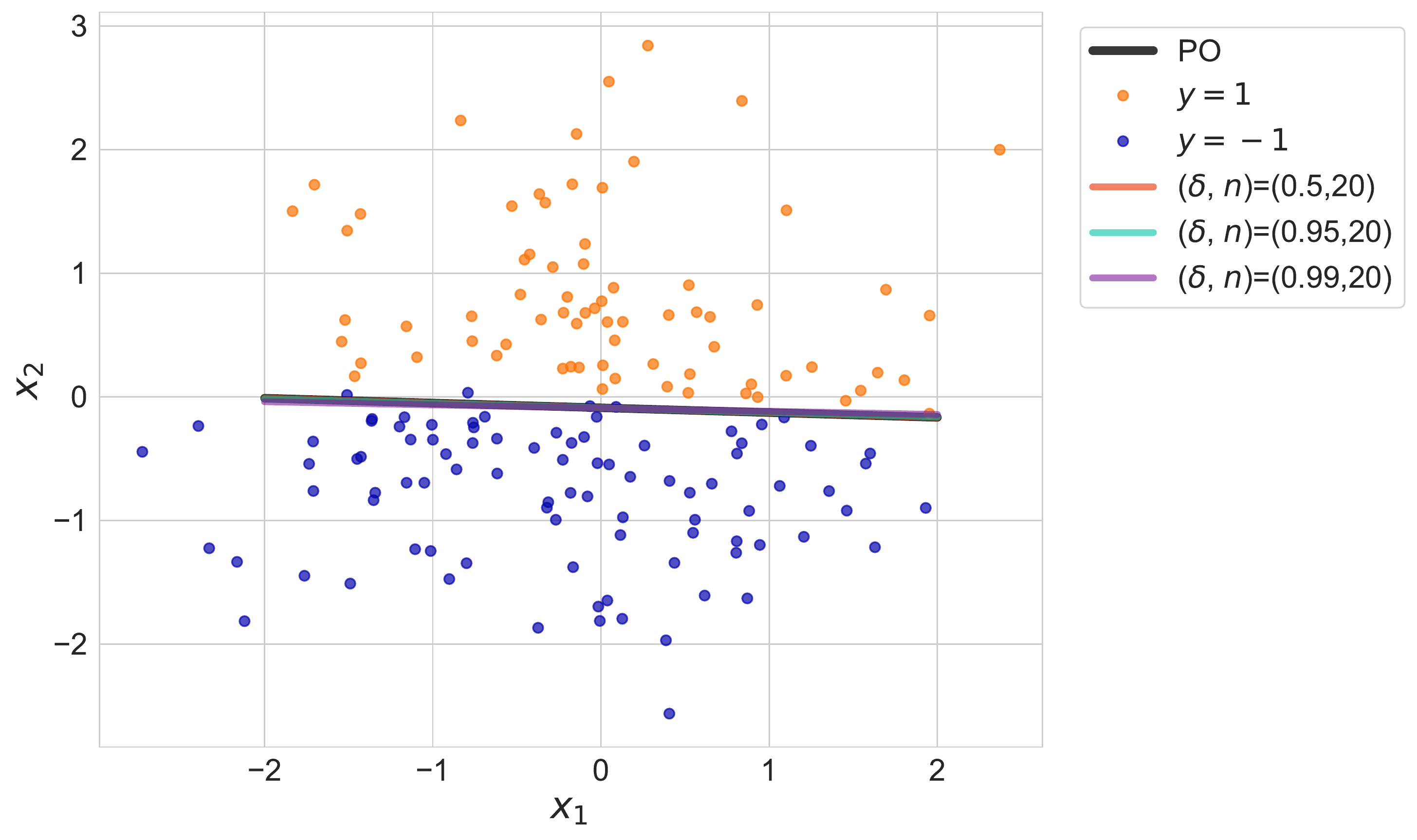}\includegraphics[width=0.32\textwidth]{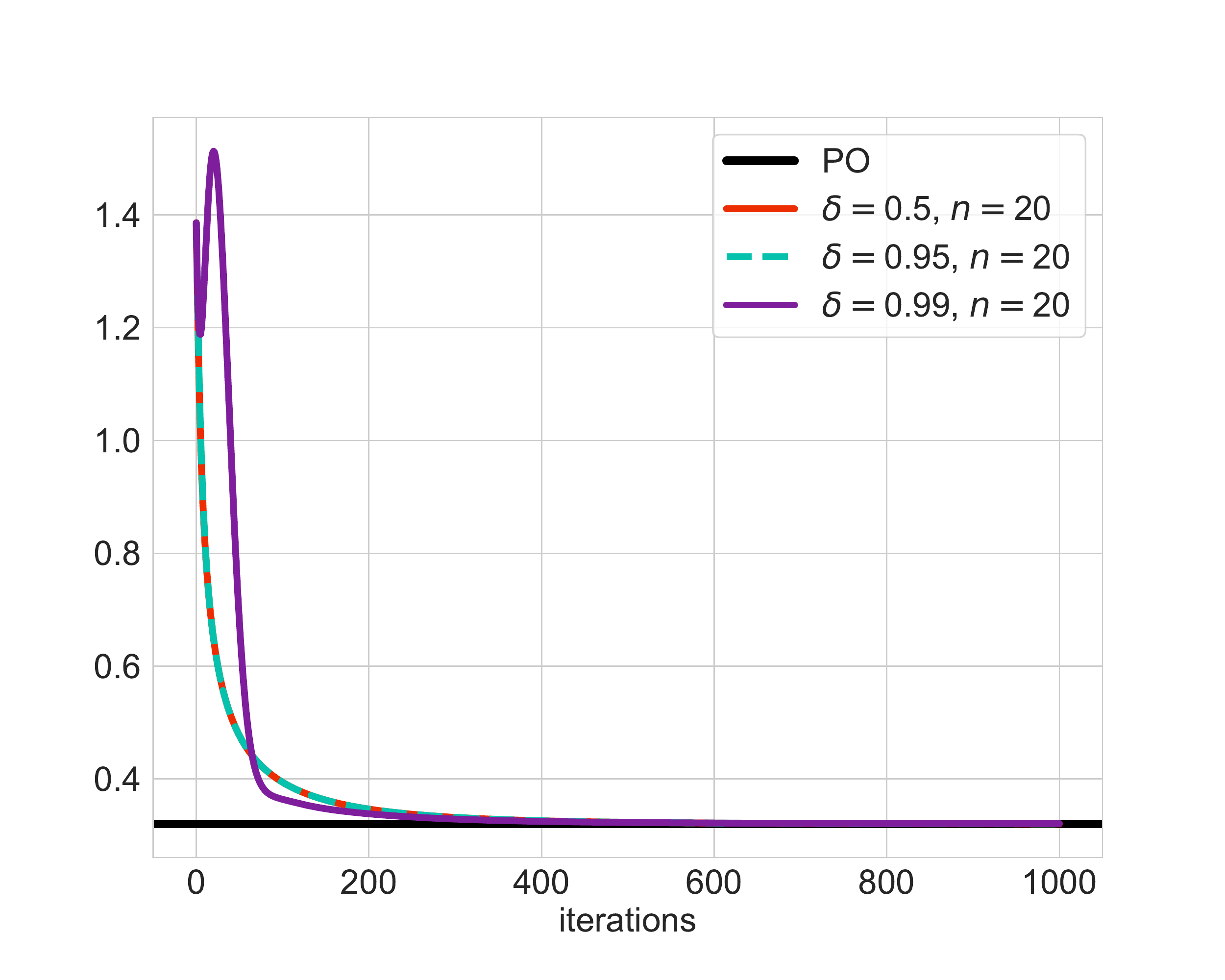}\includegraphics[width=0.29\textwidth]{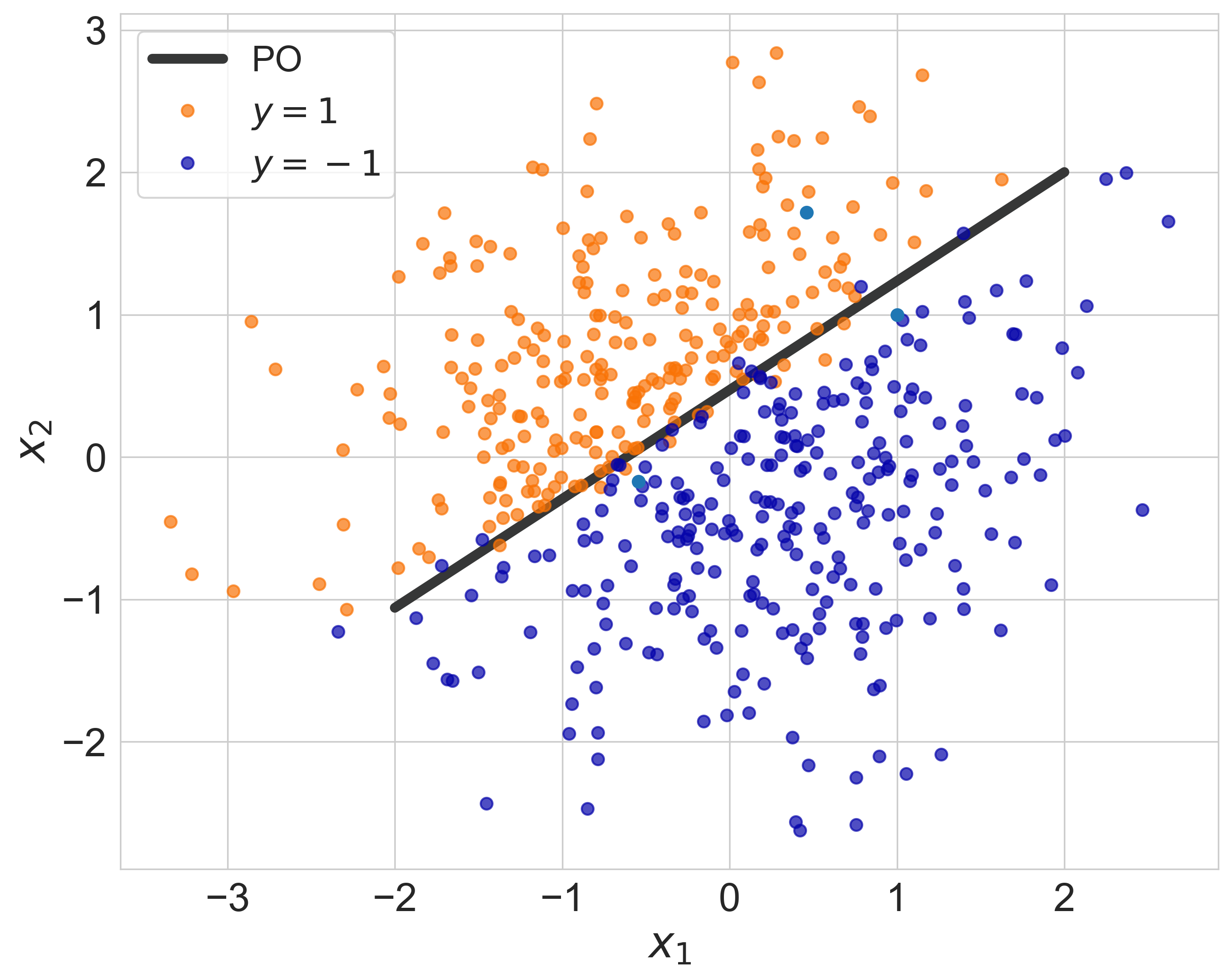}
    \caption{Classifiers and losses for different values of $\georate$ and $n$, for the given original data distribution shown in the far right plot. (\textbf{left}) Different classifiers (as a function of $\georate$ and $n$) and the data distribution given the strategic best response at the performatively optimal point. (\textbf{center}) Losses for the different $(\georate,n)$ pairs as a function of iteration. (\textbf{right}) original data distribution and ground truth classifier.}
    \label{fig:seed980classifier}
\end{figure}
\begin{figure}[h!]
    \centering
    \includegraphics[width=0.85\textwidth]{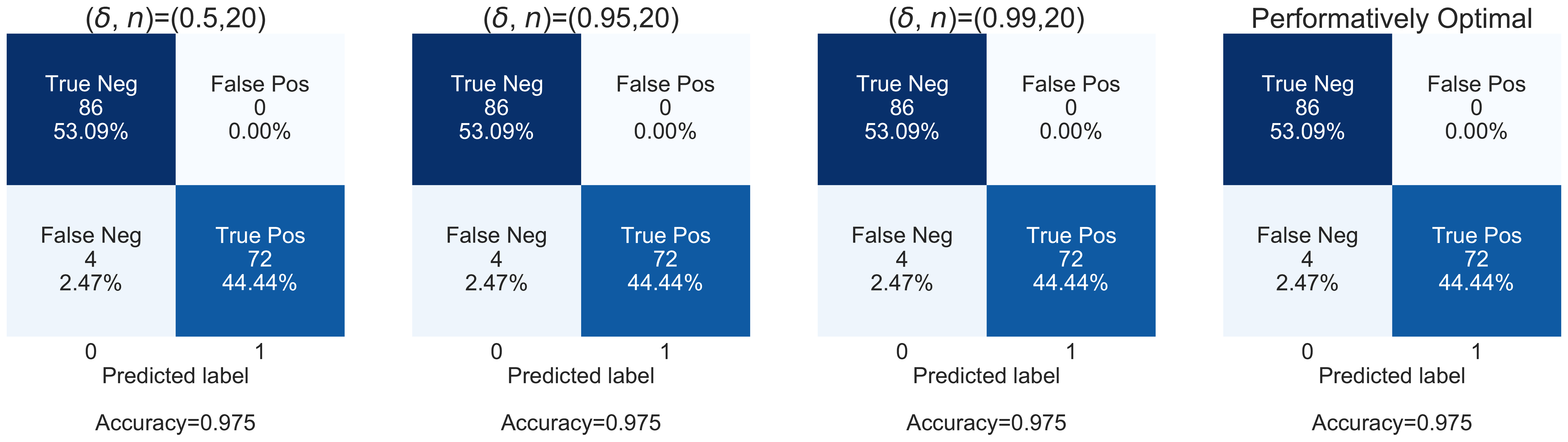}
    \caption{Accuracy of the classifiers (via confusion matrix) learned for the data distribution and setting shown in Figure \ref{fig:seed980classifier}. For this randomly sampled data distribution, the value of $\georate$  \textbf{does not} play a significant role on the generalization capability (as measured by accuracy on the test set). Accuracy remains the same across the learned classifiers in each setting.   }
    \label{fig:seed980accuracy}
\end{figure}
These observations about the generalization performance of the obtained solution
under our proposed algorithm (for different values of the geometric process or
mixing constant $\georate$) as compared to the (performatively) optimal point, while highly dependent on the underlying data distribution, open up a number of interesting directions for future work on understanding precisely when the optimal point gives good generalization and robustness guarantees.

\subsection{Semi-Synthetic Data: Strategic Classification in Dynamic Environments}
\label{app:semi-synthetic}

As a point of comparison to the existing literature, we perform additional numerical experiments on a strategic classification simulator from the Kaggle \textit{Give Me Some Credit} dataset discussed in \citet{perdomo2020performative} and \citet{brown2020performative}. In this dataset, each data point contains a feature vector, $\phi\in\mb{R}^\dimm$, which represents historical information about an individual, and the label, $y\in\{0,1\}$, which represents whether or not the individual has defaulted on a loan. For more details on the dataset itself, see Appendix B.2 in \citet{perdomo2020performative}. 

Let $S$ be the subset of features that an individual can strategically manipulate. We assume that the best response of every individual to an announced $x$ is given by $\phi^S-\tilde{\varepsilon}x^S$, where we use the notation $x^S$ to be the restriction of $x$ to the subset $S$ and similarly for $\phi^S$. The remaining features of the individual stay the same as the original data.

We conduct two sets of experiments. In the first set, we compare our algorithm on the total number of iterations---i.e., epochs $n$ multiplied by $T$---to risk minimization (RRM) \citep{perdomo2020performative,brown2020performative}, and repeated gradient descent (RGD) \citep{perdomo2020performative}---implemented for the dynamic environment which was not considered in \citet{perdomo2020performative}---both of which, notably update $x$ at every iteration in $[0,nT]$ where as our approach (Algorithm \ref{alg:main-fo}) only updates at every $n$ steps in that same interval. 

\begin{figure*}[t!]
    \centering
    \includegraphics[width=1.0\textwidth]{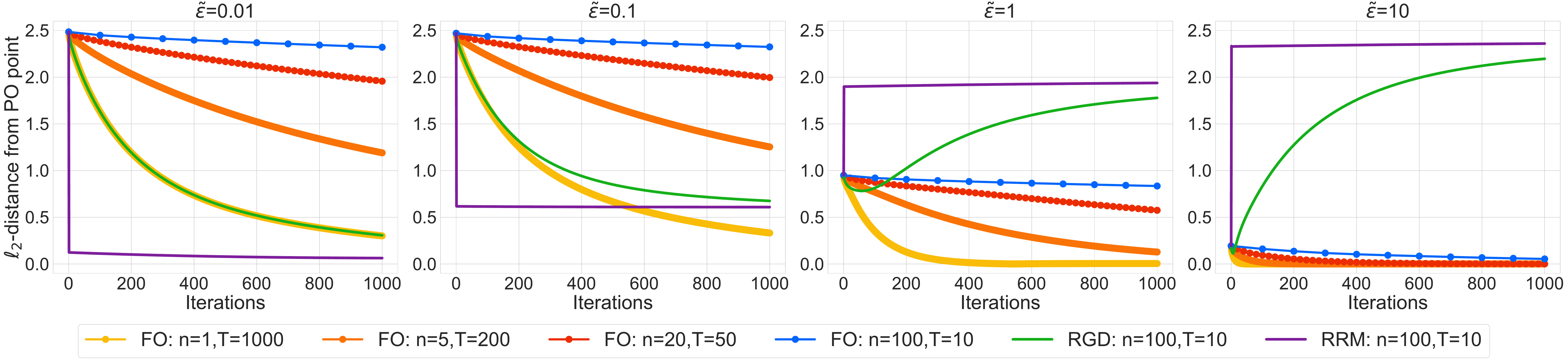}
    \caption{\textbf{`Give Me Some Credit' Experiment 1}: Results of Algorithm \ref{alg:main-fo} called with different $(n,T)$ pairs along with standard implementations of repeated risk minimization (RRM) and repeated gradient descent (RGD) wherein the dynamics and classifier are updated at each iteration.
    Each marker represents a new $x$ announcement, and the plots show the Euclidean distance from the performatively optimal point. Algorithm~\ref{alg:main-fo} converges to the performatively optimal point for each value of $\tilde{\varepsilon}$ while RRM and RGD converge to the performatively stable point. The latter may be far from the performatively optimal point for large perturbation values $\tilde{\epsilon}$ as indicated in the plot, going from left to right.}
    \label{fig:givemecredit}
\end{figure*}

In the second set of experiments, we compare our approach to an epoch based implementation of both RRM and RGD where in these implementations the dynamics are also allowed to ``mix" and the decision maker updates only every $n$ steps as in our method. These later experiments are more comparable even though the epoch based implementations of RRM and RGD have not been studied theoretically.
For both experiments, we plot the $\ell_2$ distance to the  optimal point. 

\begin{figure*}[t!]
    \centering
    \includegraphics[width=1.0\textwidth]{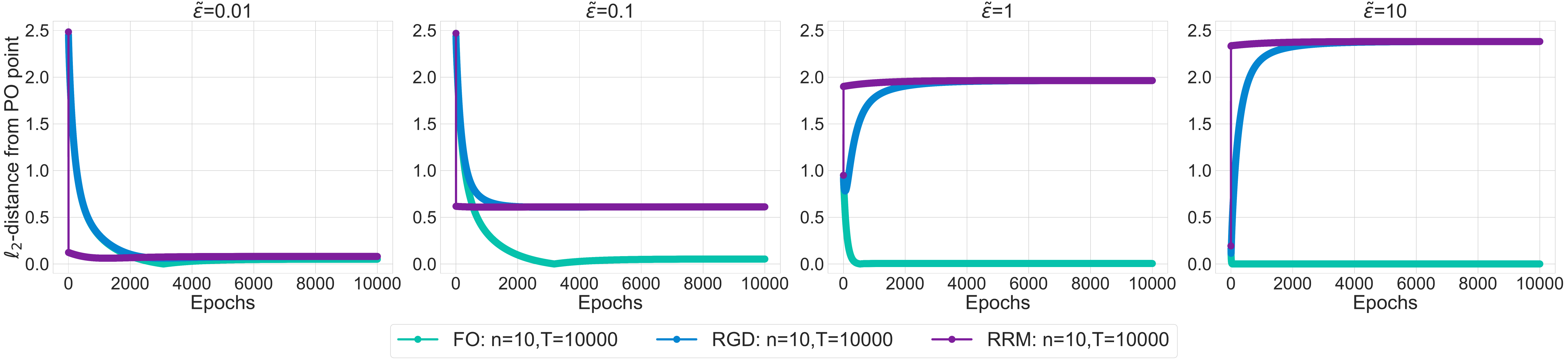}
       \includegraphics[width=1.0\textwidth]{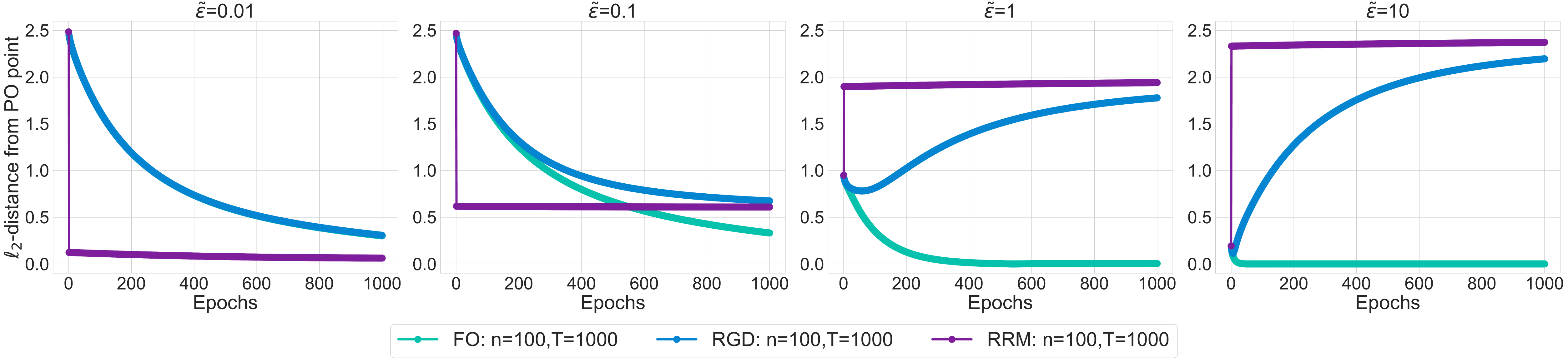}
    \caption{\textbf{`Give Me Some Credit' Experiment 2}: Results of Algorithm \ref{alg:main-fo} compared to epoch-based implementations of RRM and RGD---i.e., where in each epoch the dynamics are updated $n$ times with the same classifier deployed---each called with $(n,T)\in\{(10,1000), (100,1000)\}$.
    Each marker represents a new $x$ announcement, and the plots show the Euclidean distance from the performatively optimal point.}
    \label{fig:givemecredit-2}
\end{figure*}
\paragraph{Experiment 1: Comparison to Iteration-Based (Classical) RRM and RGD.}
Figure~\ref{fig:givemecredit} shows the results of the first set of experiments, for which we have taken $\georate=0.9$, which is relatively large meaning that the mixing time for the geometric process is large. 
Neither RRM nor RGD target the performatively optimal point, but instead the \textit{performatively stable} point, i.e., the point at which repeated retraining will stabilize. As shown in Figure \ref{fig:givemecredit}, a performatively stable point (the point RRM was shown to converge to in \citet{brown2020performative}) may be far from the peformatively optimal point. Interestingly, we also observe that for small values of $\tilde{\varepsilon}$ (i.e. on the order of $1\text{e-}2$), the performatively optimal point and the performatively stable point are very close, and so RGD behaves nearly identically to calling Algorithm \ref{alg:main-fo} with $n=1$. This seems to imply that when performative effects (i.e., size of $\tilde{\varepsilon}$ in this set of experiments) are very low, the na\"{\i}ve strategies of RRM or RGD suffice when trying to find the optimal point. On the other hand, for values of $\tilde{\varepsilon}$ on the order of $1\text{e-}1$ or larger, RRM and RGD do not converge to the performatively optimal point while Algorithm~\ref{alg:main-fo} does, albeit with worse iteration complexity to convergence to the stable point of the respective algorithm. 

\paragraph{Experiment 2: Comparison to Epoch-Based RRM and RGD.}
Figure~\ref{fig:givemecredit-2} shows the results of the second set of experiments. As noted above, in this set of experiments, we compare to epoch based implementations of RRM and RGD to Algorithm~\ref{alg:main-fo} which is also an epoch-based algorithm, the idea here being that these are more comparable algorithms in a sense.
As can be seen in Figure~\ref{fig:givemecredit-2}, the observations are analogous to the first set of experiments. Epoch-based RRM and RGD converge to the performatively stable point (as defined in \citep{perdomo2020performative} and \citep{brown2020performative}, for the dynamic setting). For $\tilde{\varepsilon}$ on the order of $1\text{e-}2$, the performatively stable point is close to the performatively optimal point (although still not equal to it), and for $\tilde{\varepsilon}$ on the order of $1\text{e-}1$ or larger, the performatively stable point is considerably farther away from the performatively optimal point. On the other hand,  Algorithm~\ref{alg:main-fo} converges to the performatively optimal point for all shown values of $\tilde{\epsilon}$, the size of the strategic perturbation.

We note that we did not compare to the zero-th order method since it has different information than both the RRM and RGD and is thus less comparable. We expect the same observations about non-convergence of RRM and RGD for large $\tilde{\varepsilon}$ to persist and Algorithm~\ref{alg:main-zo} will converge as the theory predicts, albeit at a much slower rate than Algorithm~\ref{alg:main-fo} due to the bandit feedback.

\end{document}

%% file: intro.tex
Traditionally, supervised machine learning algorithms are trained  based on past data under the assumption that the past data is representative of the future. However, machine learning algorithms are increasingly being used in settings where the output of the algorithm 
changes the environment and hence, the data distribution.
Indeed,
online labor markets \citep{anagnostopoulos2018algorithms,horton2010online}, predictive policing \citep{lum2016predict}, on-street parking \citep{pierce2018sfpark,dowling2019tits}, 
and vehicle sharing markets \citep{banerjee2015pricing} 
are all examples of real-world settings in which the algorithm's decisions change the underlying data distribution due to the fact that the algorithm interacts with  strategic users.

To address this problem, the machine learning community introduced the problem of \emph{performative prediction} which models the data distribution as being \emph{decision-dependent} thereby accounting for feedback induced distributional shift \citep{perdomo2020performative, miller2021outside,drusvyatskiy2020stochastic,brown2020performative,mendler2020stochastic}. With the exception of \citep{brown2020performative}, this work has focused on static environments. 

In many of the aforementioned application domains, however, the underlying data distribution also may have memory or even be changing dynamically in time. 
When a decision-making mechanism is announced it may take time to see the full effect of the decision as the environment and strategic data sources respond given their prior history or interactions. 

For example, many municipalities announce quarterly a new quasi-static set of prices for on-street parking. In this scenario, the institution may adjust parking rates for certain blocks in order to to achieve a desired occupancy range to reduce cruising phenomena and increase business district vitality \citep{fiez2018data,dowling2017cdc,pierce2013getting,shoup2006cruising}.   For instance, in high traffic areas, the institution may announce increased parking rates to free up parking spots and redistribute those drivers to less populated blocks. However, upon announcing a new price, the population may react slowly,  whether it be from initially being unaware of the price change, to facing natural inconveniences from changing one's parking routine. This introduces dynamics into our setting; 
hence, the data distribution takes time to equilibrate after the pricing change is made.

Motivated by such scenarios, we study the problem of decision-dependent risk minimization (or, synonymously, performative prediction) in dynamic settings wherein the underlying decision-dependent distribution evolves according to a geometrically decaying process. Taking into account the time it takes for a decision to have the full effect on the environment, we devise an algorithmic framework  for finding the optimal solution in settings where the decision maker has access to different types of gradient information. 
 
For both information settings (gradient access and loss function access, via the appropriate oracle), the decision-maker deploys the current decision repeatedly for the duration of an epoch, thereby allowing the dynamically evolving distribution to approach the fixed point distribution for that announced decision.
At the end of the epoch, the decision is updated using a first-order or zeroth-order oracle.

One interpretation of this procedure is that the environment is operating on a faster timescale compared to the update of the decision-maker's action. For instance, consider the dynamically changing distribution as the data distribution corresponding to a population of strategic data sources.
The phase during which the same decision is deployed for a fixed number of steps can be interpreted as the population of agents adapting at a faster rate than the update of the decision. This in fact occurs in many practical settings such as on-street parking, wherein prices and policies more generally are \emph{quasi-static}, meaning they are updated infrequently relative to actual curb space utilization.

\subsection{Contributions}
For the decision-dependent learning problem in geometrically decaying environments, we propose first-order or zeroth-order oracle algorithms   that converge to the optimal point under appropriate assumptions, which make the risk minimization problem strongly convex. 
We obtain the following iteration complexity guarantees:
\begin{itemize}
    \item \textbf{Zero Order Oracle}  (Algorithm~\ref{alg:main-zo}, Section~\ref{sec:algs}): We show that the sample complexity in the zeroth order setting is $\tilde{O}(\frac{d^2}{\varepsilon^{2}})$ which matches the optimal rate for single query zeroth order methods in strongly convex settings up to logarithmic factors.
    \item \textbf{First Order Oracle} (Algorithm~\ref{alg:main-fo}, Section~\ref{sec:algs}):  We show that the same complexity in the first order setting is $\tilde{O}(\frac{1}{\varepsilon})$ again matching the known rates for first order stochastic gradient methods up to logarithmic factors.
\end{itemize}
The technical novelty arises from bounding the error between the expected gradient at the fixed point distribution corresponding to the current decision and the stochastic gradient at the current distribution at time $t$.

The algorithms are applied to a set of \emph{semi-synthetic} experiments using real data from  the SFpark pilot study on the use of dynamic pricing to manage curbside parking (Section~\ref{sec:experiments}). 
The experiments demonstrate  that optimizing taking into consideration feedback-induced distribution shift even in a dynamic environment leads to the institution---and perhaps surprisingly, the user as well---experiencing lower expected cost. Moreover, there are important secondary effects of this improvement including 
increased access to parking---hence, business district vitality---and reduced circling for parking and congestion which not only saves users time, but also reduces carbon emissions \citep{shoup2006cruising}. 

A more comprehensive set of experiments is contained in Appendix~\ref{app:sfpark}, including purely synthetic simulations and other semi-synthetic simulations using the `Give Me Some Credit' data set from  \cite{kaggle}.

\subsection{Related work}
\label{sec:relatedwork}

\paragraph{Dynamic Decision-Dependent Optimization.}
As hinted above, dynamic decision-dependent optimization has been considered quite extensively in the stochastic optimization literature wherein the problem of \emph{recourse} arises due to decision-makers being able to make a secondary decision after some information has been revealed 
\citep{jonsbraaten1998class,goel2004stochastic,varaiya1988stochastic}. In this problem, the goal of the institution is to solve a multi-stage stochastic program, in which the probability distribution of the population is a function of the decision announced by the institution. This multi-stage procedure models a dynamic process.  Unlike the setting considered in this paper, the institution has the ability to make a recourse decision upon observing full or partial information about the stochastic components.

\paragraph{Reinforcement Learning.} Reinforcement learning is a more closely related problem in the sense that a decision is being made over time where the environment dynamically changes as a function of the state and the decision-maker's actions \citep{sutton2018reinforcement}. A subtle but important difference is that the setting we consider is such that the decision maker's objective is to find the action  which optimizes the decision-dependent expected risk at the fixed point distribution (cf.~Definition \ref{def:performatively_optimal_point}, Section~\ref{sec:prelims}) induced by the optimal action
and the environment dynamics. 
This is in contrast to finding a policy which is a state-dependent distribution over actions given an accumulated cost over time. Our setting can be viewed as a 
special case of the general reinforcement learning problem, 
however with additional structure that is both practically well-motivated, and beneficial to exploit in the design and analysis of algorithms. More concretely,
we crucially exploit the assumed model of environment dynamics (in this case, the geometric decay), the distribution dependence, and convexity to obtain strong convergence guarantees for the algorithms proposed herein. 

\paragraph{Performative prediction.} As alluded to in the introductory remarks, the  most closely related body of literature is on performative prediction wherein 
the decision-maker or optimizer takes into consideration that the underlying data distribution depends on the decision. %
A na\"{i}ve strategy is to  re-train the  model after using heuristics to determine when there is sufficient distribution shift. 
Under the guise that if retraining is repeated, eventually the distribution will stabilize, 
 early works on performative prediction---such as the works of \citet{perdomo2020performative} and \citet{mendler2020stochastic}---studied this equilibrium notion, and called these points
\emph{performatively stable}. \citet{mendler2020stochastic} and \citet{drusvyatskiy2020stochastic} study stochastic optimization algorithms applied to the performative prediction problem and recover optimal convergence guarantees to the performatively stable point. Yet, performatively stable points may differ from the optimal solution
of the decision-dependent risk minimization problem as was shown in \citet{perdomo2020performative}. 
Taking this gap between stable and optimal points into consideration, 
\citet{miller2021outside} characterize when the performative prediction problem is strongly convex, and devise a two-stage algorithm for finding the so-called \emph{performatively optimal} solution---that is, the optimal solution to the decision-dependent risk minimization problem---when the decision-dependent distribution is from the location-scale family. 

None of the aforementioned works consider dynamic environments. 
\citet{brown2020performative} is the first paper, to our knowledge, to investigate the dynamic setting for performative prediction.
Assuming regularity properties of the dynamics, they show that classical retraining algorithms (repeated gradient descent and repeated risk minimization) converge to the performatively stable point of the expected risk at the corresponding fixed point distribution.
Counter to this, in this paper we propose algorithms for the dynamic setting which target performatively optimal points.

%% file: prelims.tex
We consider the problem of a single decision-maker facing a decision dependent learning problem in a geometrically decaying environment. 

Towards formally defining the optimization problem the decision-maker faces, we first introduce some notation. Throughout, we let $\mb{R}^d$ denote a $d$--dimensional Euclidean space with inner product $\la\cdot,\cdot\ra$ and induced norm $\|x\|=\sqrt{\la x,x\ra}$. The projection of a point $y\in \mb{R}^d$ onto a set $\mc{X}\subset \mb{R}^d$ is denoted $\proj_{\mc{X}}(y)=\argmin_{x\in \mc{X}}\|x-y\|$. We are interested in random variables taking values in a metric space. Given a metric space $\mc{Z}$ with metric $\mathrm{d}(\cdot,\cdot)$ the symbol $\mathbb{P}(\mc{Z})$ denotes the set of Radon probability measures $\nu$ on $\mc{Z}$ with a finite first moment $\E_{z\sim \nu}[\mathrm{d}(z,z')]<\infty$ for some $z'\in \mc{Z}$. We measure the deviation between two measures $\nu,\nu'\in \mb{P}(\mc{Z})$ using the Wasserstein-1 distance:
\[W_1(\nu,\mu)=\sup_{h\in \text{Lip}_1}\{\mb{E}_{X\sim \nu}[h(X)]-\mb{E}_{Y\sim \mu}[h(Y)]\},\]
where $\text{Lip}_1$ denotes the set of $1$--Lipschitz continuous functions $h:\mc{Z}\to\mb{R}$. 

The decision-maker seeks to solve 
\begin{equation}\label{eq:opt_prob}
    \min_{x\in \mc{X}}\LL(x)
\end{equation}
where $\LL(x)=\mb{E}_{z\sim \mc{D}(x)}[\ell(x,z)]$ is the expected loss. The decision-space $\mc{X}$ lies in the Euclidean space $\mb{R}^d$, is closed and convex, and there exists constants $r,R>0$ satisfying $r\mb{B}\subseteq \X\subseteq R\mb{B}$ where $\mb{B}$ is the unit ball in dimension $d$. The loss function is denoted $\ell:\mb{R}^d\times \mc{Z}\to \mb{R}$, and $\mc{D}(x)\in \mb{P}(\mc{Z})$ is a probability measure that depends on the decision $x\in \mc{X}$.
\begin{definition}
\label{def:performatively_optimal_point}
For a given probability measure $\mc{D}(x)$ induced by action $x\in \mc{X}$, the decision vector $x^\ast\in\X$ is optimal  if
\[x^\ast\in\arg\min_{x\in\X}\LL(x)=\arg\min_{x\in\X}\E_{z\sim \mc{D}(x)}[\ell(z,x)].\]
\end{definition}
The main challenge to finding an optimal point is that the environment is 
evolving in time according to a geometrically decaying process. 
That is, the random variable $z$ depends not only on the decision $x_t\in \X$ at time $t$, but also explicitly on the time instant $t$. In particular, the random variable $z$ is governed by the distribution $\dist_t$ which is  the probability measure at time $t$
generated by the process $\dist_{t+1}=\mc{T}(\dist_t, x_t)$
where
\begin{equation}
\label{eq:dynamics}
\mc{T}(\dist,x)=\georate \dist+(1-\georate)\mc{D}(x),
\end{equation}
and $\georate\in[0,1)$ is the geometric decay rate. 
Observe that given the geometrically decaying dynamics in  \eqref{eq:dynamics},
for any $x\in \X$, the distribution $\mc{D}(x)$ is trivially a fixed point---i.e., $\mc{T}(\mc{D}(x),x)=\mc{D}(x)$. 
Let 
$\mc{T}^n:=\mc{T}\circ\dots\circ\mc{T}$
denote the $\epoch$-times composition of the map $\mc{T}$ so that, given the form in \eqref{eq:dynamics}, we have $\mc{T}^n(p,x)=\georate^np+(1-\georate^n)\mc{D}(x)$. 

One interpretation of this transition map is that it captures the phenomenon that for each time, a $(1-\georate)$ fraction of the population becomes aware of the machine learning model $x$ being used by the institution. Another interpretation is that the environment (and strategic data sources in the environment) has memory based on past interactions which is captured in the `state' of the distribution, and the effects of the past decay geometrically at a rate of $\georate$. For instance, it is known in behavioral economics that humans often compare their decisions to a reference point, and that reference point may evolve in time and represent an accumulation of past outcomes~\citep{nar2017learning,kahneman2013prospect}.

Throughout we use the notation $\nabla \LL$ to denote the derivative of $\LL$ with respect to $x$. 
 The notation $\nabla_x\ell$ and $\nabla_z\ell$ denotes the partial derivative of $\ell$ with respect to $x$ and $z$, respectively. Further, let $\nabla_{x,z}\ell=(\nabla_x\ell,\nabla_z\ell)$ denote the vector of partial derivatives.
We also make the following standing assumptions on the loss $\ell$ and the probability measure $\mc{D}(x)$.
\begin{assumption}[Standing]\label{a:standing}
The loss $\ell$ and distribution $\mc{D}$ satisfy the following:
\begin{enumerate}[label={\alph*.}]
    \item The loss  $\ell(x,z)$ is $C^1$ smooth in $x$, and $L$-Lipschitz continuous in $(x,z)$. 
 \item The map $(x,z)\mapsto\nabla_{x,z}\ell(x,z)$ is $\beta$-Lipschitz continuous.
\item The loss $\ell(x,z)$  is $\xi$-strongly convex in $x$.
\item There exists a constant $\gamma>0$ such that  \[{W}_1(\mc{D}(x),\mc{D}(x'))\leq\gamma \|x-x'\|\quad \forall\ x,x'\in \X.\]
\end{enumerate}
\end{assumption}

The following assumption implies a convex ordering on the random variables on which the loss is dependent. 
\begin{assumption}[Mixture Dominance]\label{a:mixture_dominance}
The probability measure $\mc{D}(x)$ and loss $\ell$ satisfy mixture dominance---i.e.,  for any $x\in \X$ and $s\in (0,1)$, 
\[\E_{z\sim \mc{D}(sv+(1-s)w)}[\ell(z,x)]\leq \E_{z\sim s\mc{D}(v)+(1-s)\mc{D}(w)}[\ell(z,x)],\quad \forall \ v,w\in \X.\]
\end{assumption}

Under Assumptions \ref{a:standing} and \ref{a:mixture_dominance}, the expected loss $\LL(x)$ is $\alpha:=(\xi-2\gamma\beta)$ strongly convex (cf.~Theorem 3.1 \citet{miller2021outside}), and so the optimal point is unique. 

We make the following assumption on the regularity of the expected loss. 
\begin{assumption}[Smoothness]  The map $x\mapsto \nabla \LL(x)$ is $\LLlip$-Lipschitz continuous, and the map $x\mapsto\nabla^2\LL(x)$ is $H$-Lipschitz continuous.
\label{a:smooth_ell}
\end{assumption}
An important class of distributions in the performative prediction literature that satisfy this assumption are location-scale distributions. 
\begin{assumption}[Parametric family]\label{a:locationscale_main}
There exists a probability measure $\mc{P}$ and matrix $A$ such that
\[z\sim \mc{D}(x)\ \Longleftrightarrow\ z=\zeta+A x, \]
and where $\zeta$ has mean $\mu:=\E_{\zeta\sim\mc{P}}[\zeta]$ and co-variance $\Sigma:=\E_{\zeta\sim\mc{P}}[(\zeta-\mu)(\zeta-\mu)^\top]$, respectively.
\end{assumption}
This class encompasses a broad set of distributions that are commonplace in the performative prediction literature. As observed in \citet{miller2021outside}, this class of probability measures is also $\gamma$-Lipschitz continuous and satisfies the mixture dominance condition when $\ell$ is convex. 
\begin{restatable}[Sufficient conditions for Assumption~\ref{a:smooth_ell}]{lemma}{lemmaparametric}\label{lem:lemmaparametric}
 Suppose that Assumption~\ref{a:locationscale_main} holds and there exists constants $\beta,\rho\geq 0$ such that the map $(x,z)\mapsto\nabla_{x,z}\ell(x,z)$ is $\beta$-Lipschitz continuous and has a $\rho$-Lipschitz continuous gradient. Then, Assumption~\ref{a:smooth_ell} holds with constants
\[\LLlip:=\sqrt{\beta^2\max\{1,\|A\|_{\op}^2\}\cdot (1+\|A\|_{\op}^2)},\]
\[H:=\sqrt{\rho^2\max\{1,\|A\|_{\op}^4\}\cdot (1+\|A\|_{\op}^2)}.\]
\end{restatable}
The proof is contained in Appendix~\ref{app:techlemmas}.

%% file: results_empirical.tex
\subsection{Zero Order Stochastic Gradient Method}
\label{sec:zoempirical}
The most general information setting we consider is such that the decision-maker has only ``bandit feedback". That is, they only have access to a loss function evaluation oracle. This does not require the decision-maker to have access to the decision-dependent  probability measure $\mc{D}(x)$. This is a  more realistic setting given that the form of $\mc{D}(\cdot)$ may be a priori unknown. 
For example, if the data is  generated by strategic data sources having their own private utility functions and preferences (e.g., as in strategic classification or prediction, or incentive/pricing design problems), then the decision-maker does not necessarily have access to the distribution map $\mc{D}(x)$ in practice.

The zero-order stochastic gradient method proceeds as follows. Fix a parameter $\delta>0$. 
In each epoch $t$, Algorithm~\ref{alg:main-zo} samples $v_t$, a unit vector, uniformly from the unit sphere $\mb{S}$ in dimension $d$, queries the environment for $n_t$ iterations with $x_t+\delta v_t$, and then receives feedback from the loss oracle which reveals $\ell(x_t+\delta v_t,z_t)$ where $z_t\sim \georate^{n_t}\dist_{t-1}+(1-\georate^{n_t})\mc{D}(x_t+\delta v_t)$ which the decision maker uses to update $x_t$ as follows:
\[x_{t+1}=\proj_{(1-\delta)\X}\left(x_t-\eta_t \hat{g}_t \right),\]
where
\begin{equation}\label{eq:zero_order}
   \hat{g}_t=\frac{\dimm}{\delta}\ell(x_t+\delta v_t,z_t)v_t.
\end{equation}
This is a one-point gradient estimate of the expected loss at $\dist_t$. 
It can be shown that \eqref{eq:zero_order} is an unbiased estimate of the gradient of the smoothed loss function
\[\mc{L}_t^{\delta}(x)=\E_{v\sim \mb{S}}\left[\E_{z\sim p_t}\ell(x,z)\right]\]
at time $t$
 (e.g., in the general setting without decision-dependence this follows from~\citet{flaxman2004online}).
The reason for projecting onto the set  $(1-\delta)\X$ is to ensure that in the next iteration, the decision is in the feasible set.

Define the smoothed expected risk as follows:
\begin{equation*}
    \LL^\delta(x)=\E_{v\sim\mb{B}}\left[\E_{z\sim \mc{D}(x+\delta v)}[\ell(x+\delta v,z)]\right].
\end{equation*}
It is straightforward to show that $\LL^\delta$ is strongly convex with parameter $(1-c)\alpha$ for some $c\in (0,1)$ in the regime where $\delta\leq c\alpha/H$ (cf.~Lemma~\ref{lem:strong_convex_mu}, Appendix~\ref{app:zero_order}).

To obtain convergence guarantees we need the following additional assumption.
\begin{assumption}
The quantity
$\ell_\ast:=\sup\{|\ell(x,z)|:\ x\in \X, \ z\in \Z\}$
is finite.
\label{a:ell_bdd}
\end{assumption}
The next lemma provides a crucial step in the proof of our main convergence result for the bandit feedback setting: it provides a bound on the bias due to the dynamics.
\begin{restatable}{lemma}{lemmaexpectedboundedgradient}
\label{lemma:expected_bounded_gradient} 
Under Assumptions \ref{a:standing}, \ref{a:mixture_dominance}, \ref{a:smooth_ell},
and \ref{a:ell_bdd}, the  error between the gradient smoothed loss $\LL_t^\delta$ at $\dist_t$ and the gradient of the smoothed expected loss $\LL^\delta$ satisfies
\begin{align*}
   & \|\nabla \mb{E}_{v\sim\mb{B}}[\mb{E}_{z\sim \dist_t}[\ell(z,x_t+\delta v)]]-\nabla\LL^\delta(x_t)\|\leq L\cdot\left(\georate^{\epoch_{t}}\overline{W}(\dist_0)+\georate^{\epoch_{t}}\frac{4\gamma\dimm}{\alpha\delta}\frac{\georate\ell_\ast}{(1-\lambda)^2}\right)
\end{align*}
where $\dist_t=\georate^{\epoch_t} \dist_{t-1}+(1-\georate^{\epoch_t})\mc{D}(x_t+\delta v_t)$, and 
$\overline{W}(\dist_0)=\max_{x\in\X}{W}_1(\dist_0,\mc{D}(x))$.
\end{restatable}
We defer the proof to Appendix \ref{app:proof_lemma_expected_bounded_gradient}.

To obtain the convergence rate, let $\bdelta$ be the optimal point for $\LL^\delta$ on $(1-\delta)\X$. 
\begin{restatable}{theorem}{theoremzeroordernew}\label{thm:zo_convergence_empirical}
Suppose that Assumptions~\ref{a:standing}, \ref{a:mixture_dominance}, \ref{a:smooth_ell}, and \ref{a:ell_bdd} hold.
Let $\delta\leq \min\{r,\frac{\alpha}{2H}\}$, and set step size $\eta_t=\frac{4}{\alpha t}$ and epoch length
\[n_t\geq \log\left(\frac{\overline{W}(\dist_0)+\frac{4\gamma\dimm}{\alpha\delta}\frac{\georate\ell_\ast}{(1-\lambda)^2}}{\left(\eta_t\frac{\alpha}{L^2}\frac{\ell_\ast^2d^2}{4\delta^2}\right)^{1/2}}\right)\frac{1}{\log(1/\georate)}.\]
 Then the estimate holds:
\begin{align*}
    \E\|x_t-x^\ast\|^2
    &\leq \frac{\max\{\alpha^2\delta^2\|x_1-\bdelta\|^2,16d^2\ell_\ast^2\}}{t\alpha^2\delta^2}+2\delta^2\left(\left(1+\frac{\LLlip}{\alpha}\right)\|x^\ast\|+\frac{\LLlip}{\alpha}\right)^2.
\end{align*}
\end{restatable}
The following corollary states the convergence rate.
\begin{restatable}[Main result for zero-order oracle]{corollary}{corollaryzeroorderrate}\label{cor:zo_convergence}
Suppose the assumptions of Theorem~\ref{thm:zo_convergence_empirical} hold. Fix a target accuracy
\[\textstyle                              \varepsilon<4r^2((1+\frac{\LLlip}{\alpha})R+\frac{\LLlip}{\alpha})^2,\]
and set $\delta=\alpha\sqrt{\varepsilon/4}/((\alpha+\LLlip)R+\LLlip)$ and $\eta_t=4/(\alpha t)$. Then, the estimate $\E\|x_t-x^\ast\|^2\leq \varepsilon$ holds for all 
\[t\geq \frac{\max\{8\alpha^2\varepsilon R^2,128 ((\alpha+\LLlip)R+\LLlip)^2\ell_\ast^2d^2\}}{\alpha^4\varepsilon^2}.\]
\end{restatable}

In the proceeding corollary, the lower bound on $t$ is in terms of the number of epochs that Algorithm~\ref{alg:main-zo} needs to be run to obtain the target accuracy. In terms of total iterations across all epochs (i.e., $\sum_{k=1}^tn_k$), the rate is thus $O\left(\frac{d^2}{\varepsilon^2}\log\left(\frac{1}{\varepsilon}\right)\right)$.

\subsection{First Order Stochastic Gradient Method}
In many situations, the decision maker has access to a parametric description of the decision-dependent probability measure $\mc{D}(x)$ in which case the decision-maker can employ a stochastic gradient method. The challenge of having the distribution changing in time still remains, and hence the novelty of the results in this section. 

To this end, let the expected loss at time $t$ be given by
\begin{equation}\label{eq:loss_time_t}
    \LL_t(x)=\E_{z\sim p_t}\ell(x_t,z).
\end{equation}
Under Assumption~\ref{a:locationscale_main} and mild smoothness assumptions, differentiating \eqref{eq:loss_time_t} we see that the gradient of $\LL_t$ is simply
\begin{equation*}
    \nabla \LL_t(x)=\E_{z\sim p_t}[\nabla_x\ell(x,z)+(1-\georate^n)A^\top \nabla_z\ell(x,z)].
\end{equation*}
Therefore, given a point $x$, the decision-maker may draw $z\sim p_t$ and form the vector
\begin{equation*}
    \hat{g}_t=\nabla \ell(x_t,z)=\nabla_x\ell(x_t,z)+(1-\georate^{n})A^\top \nabla_z\ell(x_t,z).
\end{equation*}
By definition, $\hat{g}_t$ is an unbiased estimator of $\nabla \LL_t(x)$, that is 
\[\E_{z\sim p_t}[\hat{g}_t]=\nabla\LL_t(x).\]
Algorithm~\ref{alg:main-fo} proceeds as follows. In round $t$, the decision maker queries the environment with $x_t$ for $n$ steps so that $p_t=\georate^{n}p_{t-1}+(1-\georate^{n})\mc{D}(x_t)$. Then, the gradient oracle reveals $\hat{g}_t$ as defined above where $z\sim p_t$, and the decision maker updates $x_t$ using $x_{t+1}=\proj_{\mc{X}}(x_t-\eta_t\hat{g}_t)$.

The following lemma is completely analogous to Lemma~\ref{lemma:expected_bounded_gradient}, and provides a bound on the gradient error due to the dynamics.
\begin{restatable}{lemma}{lemmaboundedgradient}
\label{lemma:bounded_gradient} 
Under Assumptions \ref{a:standing}, \ref{a:mixture_dominance}, and \ref{a:locationscale_main}, the gradient error satisfies
\begin{align*}
    &\|\nabla \mb{E}_{z\sim \dist_t}[\ell(z,x_t)]]-\nabla {\LL}(x_t)\|^2\leq L^2\cdot\left(\georate^{\epoch}\overline{W}_1(\dist_0)+\georate^{\epoch}\gamma\eta_1 \frac{L(1+\|A\|_{\op})\georate}{(1-\georate)^2}\right)^2
\end{align*}
where $\dist_t=\georate^{\epoch} \dist_{t-1}+(1-\georate^{\epoch})\mc{D}(x_t)$.
\end{restatable}
We defer the proof to Appendix \ref{sec:proof_lemma_bounded_gradient}.

\begin{assumption}[Finite Variance]\label{a:finite_variance}  There exists a constant $\sigma>0$ satisfying
\[\E_{z\sim p_t}[\|\hat{g}_t-\E_{z'\sim p_t}\nabla \ell(x,z')\|^2]\leq \sigma^2\quad \forall x\in \X, \ \forall t\geq 1.\]
\end{assumption}
To justify the above assumption, we provide sufficient conditions for the above assumption to hold in terms of the variance of the partial gradients $\nabla_{x,z}\ell$.
\begin{lemma}[Sufficient Conditions for Assumption~\ref{a:finite_variance}]\label{lem:sufficient_variance}
Suppose there exists constants $s_1,{s}_2\geq 0$ such that for all $x\in \X$ the estimates hold:
\begin{align*}
    &\E_{z\sim p_t}\|\nabla_x\ell(x,z)- \E_{z'\sim p_t}\nabla_x \ell(x,z')\|^2\leq s^2_1\\
    &\E_{z\sim p_t}\|\nabla_z\ell(x,z)- \E_{z'\sim p_t}\nabla_z \ell(x,z')\|^2\leq {s}_2^2
\end{align*}
Then Assumption~\ref{a:finite_variance} holds with $\sigma^2_t=2(s_1^2+\|A\|_{\op}^2{s}_2^2)$.
\end{lemma}

\begin{restatable}{theorem}{theoremfirstorder}
\label{thm:fo_convergence} 
Suppose that Assumptions \ref{a:standing}, \ref{a:mixture_dominance}, \ref{a:smooth_ell}, and \ref{a:locationscale_main}
hold. For step-size $\eta=\eta_t\leq \frac{\alpha}{2\LLlip^2}$ and epoch length
\[n\geq \log\left(L\frac{\overline{W}_1(p_0)+\gamma\eta L(1+\|A\|_{\op})\frac{\georate}{(1-\georate)^2}}{(\alpha\eta)^{1/2}\sigma}\right)\frac{1}{\log(1/\georate)},\]
the estimate holds:
\[\mb{E}\|x_{t+1}-x^\ast\|^2\leq \frac{1}{1+\eta\alpha}\mb{E}\|x_t-x^\ast\|^2+\frac{4\eta^2\sigma^2}{1+\eta\alpha}.\]
\end{restatable}
We defer the proof to Appendix \ref{app:fo_convergence_empirical}.  
Applying a step-decay schedule on $\eta$ yields the following corollary, the proof of which follows directly from the recursion in Theorem~\ref{thm:fo_convergence} and generic results on step decay schedules~(see, e.g., \citet[Lemma B.2]{drusvyatskiy2020stochastic}).

\begin{corollary}[Main result for first order oracle]
\label{cor:main_fo}
Suppose the assumptions of Theorem~\ref{thm:fo_convergence} hold, and that Algorithm~\ref{alg:main-fo} is run in super-epochs indexed by $k=1,\ldots, K$ wherein each super-epoch is run for $T_k$ epochs  with constant step-size $\eta_k=\frac{\alpha}{2\LLlip^2}\cdot2^{-k}$, and such that the last iterate of super-epoch $k$ is used as the first iterate in super-epoch $k+1$. Fix a target accuracy $\varepsilon>0$ and suppose $R>\|x_1-x^\ast\|^2$ is available. Set
\[\textstyle T_1=\left\lceil\frac{2}{\alpha \eta_1}\log(\frac{2R}{\varepsilon})\right\rceil, \ T_k=\left\lceil\frac{2\log(4)}{\alpha\eta_k}\right\rceil, \ \ \text{for}\ k\geq 2,\]
and $K=\lceil1+\log_2(\frac{2\eta_1\sigma^2}{\alpha\varepsilon})\rceil$. The final iterate $x$ produced satisfies $\E\|x-x^\ast\|^2\leq \varepsilon$, while the total number of epochs is at most
\[O\left(\frac{\LLlip^2}{\alpha^2}\log\left(\frac{2R}{\varepsilon}\right)+\frac{\sigma^2}{\alpha^2\varepsilon}\right).\]
\end{corollary}
It is straightforward to show that the total number of iterations is $O\left(\frac{\LLlip^2}{\alpha^2}\log\left(\frac{2R}{\varepsilon}\right)+\frac{\sigma^2}{\alpha^2\varepsilon}\log(\frac{1}{\varepsilon})\right)$.

%% file: experiments.tex
\begin{figure*}[t!]
    \centering
    \includegraphics[scale=0.16]{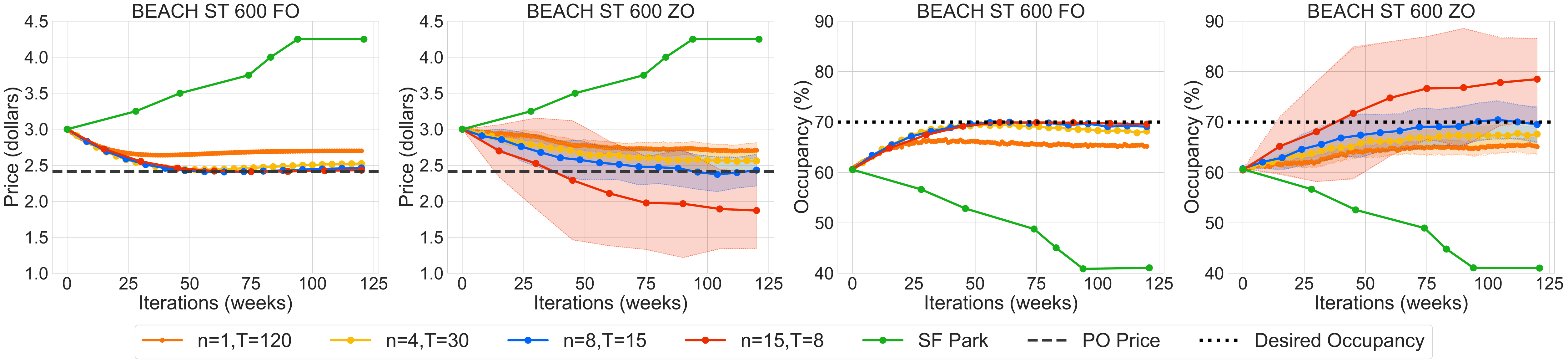}
    \caption{Results of Algorithm \ref{alg:main-fo} (first and third plots) and  Algorithm \ref{alg:main-zo} (second and fourth plots) with different $(n,T)$ pairs for $600$ Beach ST and time window $1200$--$1500$.
    Each marker represents a price announcement, and the plots show the prices and corresponding predicted occupancies. The SFpark prices and occupancies are far from the target and performative optimal price, whereas the proposed algorithms obtain both points up to theoretical error bounds. } 
    \label{fig:beachst_fo_all}
\end{figure*}

\begin{figure*}
    \centering
    \includegraphics[width=0.725\textwidth]{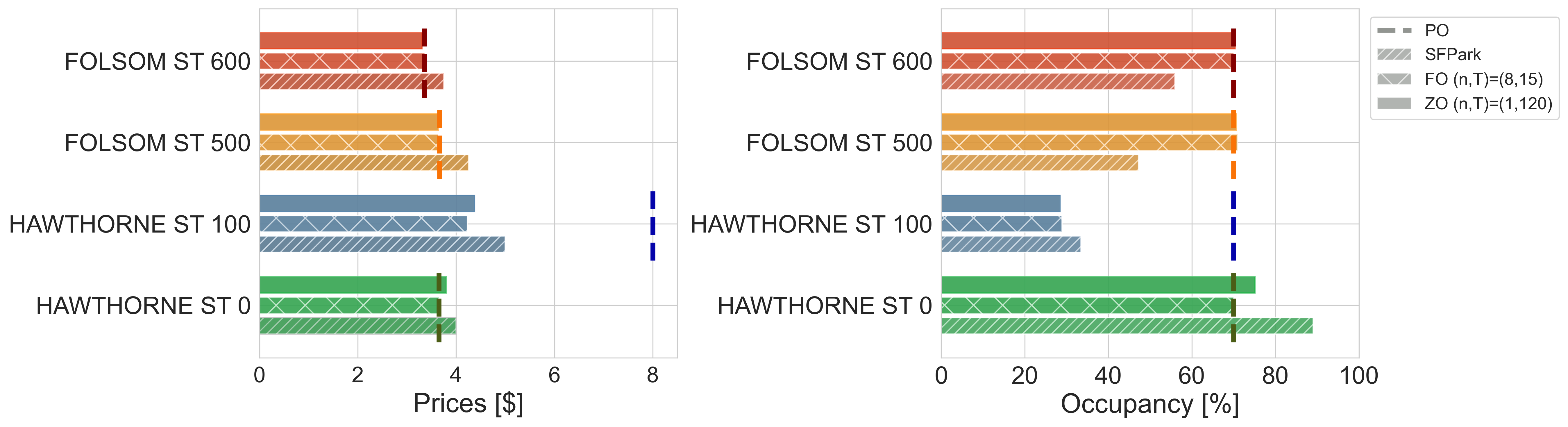}\includegraphics[width=0.225\textwidth]{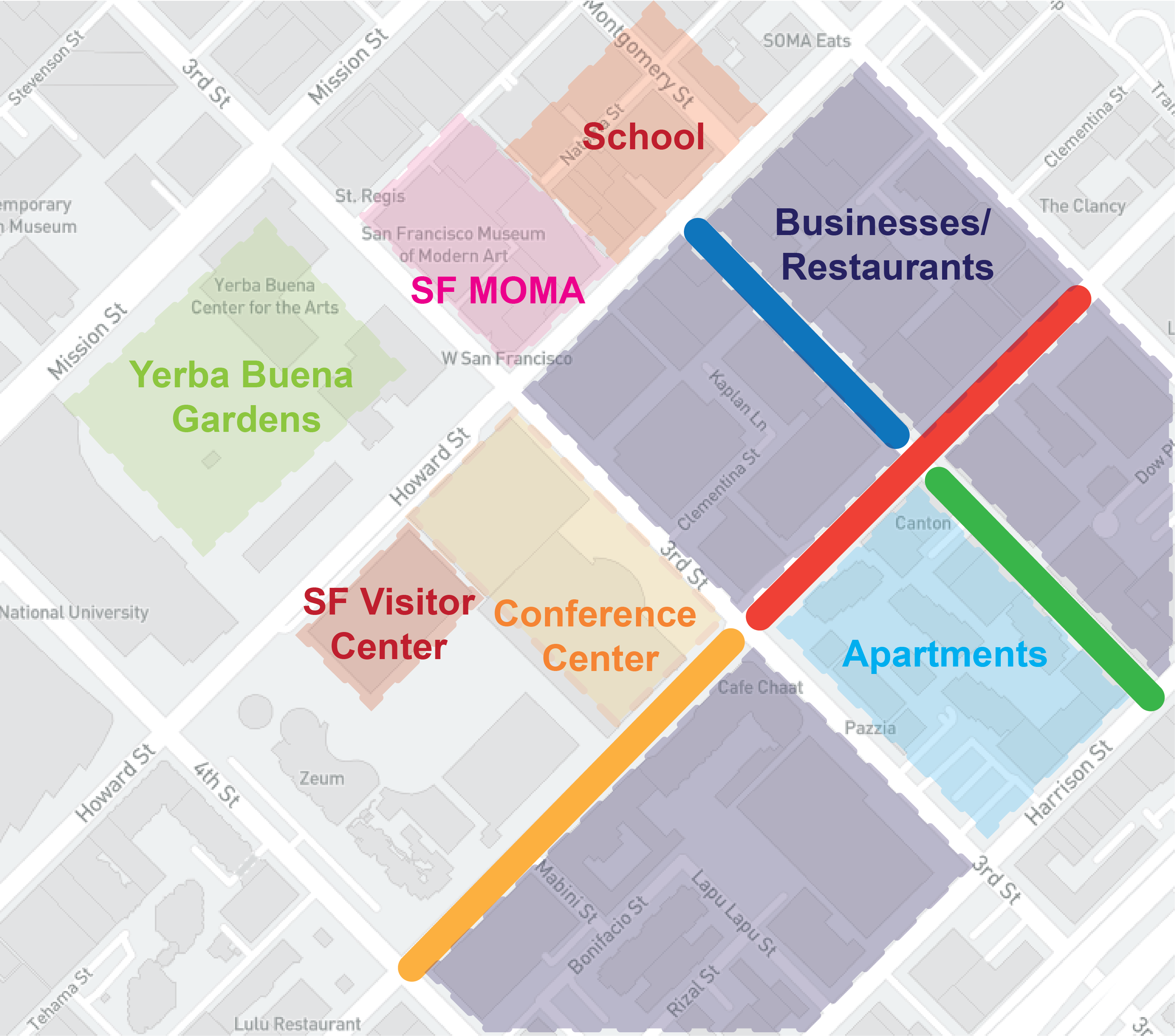}
    \caption{Final prices announced by first and zero order algorithms (Algorithms \ref{alg:main-fo} and \ref{alg:main-zo}) run with $(n,T)=(8,15)$ and $(n,T)=(1,120)$, respectively, as compared to  SFpark for streets depicted in the right graphic (color coded to the bar charts) during the $900$--$1200$ time period. The  center plot shows the corresponding predicted occupancies. The dotted lines  represent performatively optimal price and target occupancy of $70$\%, in the left and center plots, respectively. The average price overall is lower for both proposed methods,  the occupancy is better distributed, and the average occupancy closer to the desired range.  }
    \label{fig:bar}
\end{figure*}

In this section, we apply our aforementioned algorithms to 
a semi-synthetic example based on real data from the dynamic pricing experiment---namely, SFpark\footnote{SFpark: {\tt tinyurl.com/dwtf7wwn}}---for on-street parking in San Francisco. Parking availability, location, and price
are some of the most important factors when people choose
whether or not to use a personal vehicle to make a trip \citep{shoup2006cruising, shoup2021high, fiez2020tits}.\footnote{Code: {\tt\url{https://github.com/ratlifflj/D3simulator.git}}}
The primary goal of the SFpark pilot project was to make it easy to find a
parking space. To this end, SFpark  targeted a range of $60$--$80$\% occupancy in order to ensure some availability at any given time, and devised a controlled experiment for  demand responsive pricing.  
 Operational hours are split into distinct rate periods, and rates are adjusted on a block-by-block basis, using occupancy data from parking sensors in on-street parking spaces in the 
pilot areas. We focus on weekdays in the numerical experiments; for weekdays, distinct rate periods are $900$--$1200$, $1200$--$1500$, and $1500$--$1800$. Excluding special events, SFpark adjusted hourly rates  as follows:  a) $80$--$100$\% occupancy, rates are increased by \$$0.25$; b) $60$--$80$\% occupancy, no adjustment is made; c) $30$--$60$\% occupancy, rate is decreased by \$$0.25$; d) occupancy below $30$\%, rate is decreased by \$$0.50$. When a price change is deployed it takes time for users to become aware of the price change through signage and mobile payment apps~\citep{pierce2013getting}.

Given the target occupancy,  the dynamic decision-dependent loss  is given by
\[\mb{E}_{z\sim \dist_t}[\ell(x,z)]=\mb{E}_{z\sim \dist_t}[\|z-0.7\cdot\mathbf{1}\|^2+\tfrac{\nu}{2}\|x\|^2],\]
 where $z$ is the vector of curb occupancies  (which is between zero and one), $x$ is the vector of changes in price from the nominal price at the beginning of the SFpark study for each curb, and  $\nu$ is the regularization parameter. For the initial distribution $\dist_0$, we sample from the data at the beginning of the pilot study where the price is at the nominal (or initial) price. 
The distribution $\mc{D}(x)$ is defined as follows:
\[z\sim \mc{D}(x) \ \Longleftrightarrow\ z=\zeta+A x\]
where $\zeta$ follows the same distribution as $\dist_0$ described above, and $A$ is a proxy for the price elasticity which is estimated  by fitting a line to the final and initial occupancy and price (cf.~Appendix~\ref{app:sfpark-desc}).\footnote{Price elasticity is the change in percentage occupancy for a given  percentage change in price.} %

\paragraph{Comparing Performative Optimum to SFpark.}

We run Algorithms~\ref{alg:main-zo} and \ref{alg:main-fo}  for Beach ST $600$, a representative block in the Fisherman's Wharf sub-area, in the time window of $1200$--$1500$ as depicted in Figure \ref{fig:beachst_fo_all}. Beach ST  is frequently visited by  tourists and local residents. For Beach ST $600$, we compute $A\approx-0.157$, which means that a $\$1.00$ increase in the parking rate will lead to a $15\%$ decrease in parking occupancy at the fixed point distributions. Additionally, we use the data to compute the geometric decay rate of $\georate\approx0.959$ (computations described in Appendix \ref{app:sfpark}). Since the initial price is \$$3$ per hour for this block, we take $\X=[-3,5]$, since the maximum price that SFpark charges is \$$8$ per hour, and the minimum price is zero dollars. Additionally, we set the regularization parameter $\nu=1\text{e-}3$. The algorithms are run using parameters as dictated by Theorems \ref{thm:zo_convergence_empirical} and \ref{thm:fo_convergence}, respectively, with the exception of epoch length. The epoch length we set to reasonable values as dictated by the parking application. In particular, the unit of time for an iteration is weeks, and we set the epoch length in terms of the number of weeks the price is held fixed. For instance, the SFpark study changed prices every eight weeks.\footnote{In Appendix~\ref{app:sfpark}, we run synthetic experiments wherein the epoch length is chosen according to the theoretical results.}

The first and third plots in Figure \ref{fig:beachst_fo_all} show prices announced and corresponding occupancy, respectively, for Algorithm \ref{alg:main-fo}, on $600$ Beach Street, with different choices of $n$ and $T$; and, they show the prices announced and corresponding occupanices by SFpark as compared to  the performatively optimal point (computed offline).  
Similarly, the second and fourth plots in Figure \ref{fig:beachst_fo_all} show this same information for Algorithm \ref{alg:main-zo}. Since Algorithm \ref{alg:main-zo} is zero order, convergence requires more time and has variance coming from the randomness of the query directions.

SFpark  changed prices approximately every eight weeks.  
As observed in Figure \ref{fig:beachst_fo_all}, this choice of $n$ is reasonable---the estimated $\georate$ value is close to one---and leads to convergence to the optimal price change for both the first order and zero order algorithms. As $n$ increases, the performance degrades, an observation that holds more generally for this curb. However, in our experiments, we found that different curbs had different optimal epoch lengths, thereby suggesting that a non-uniform  price update schedule may lead to better outcomes. Appendix \ref{app:sfpark-compare-to-SFpark} contains additional experiments.

Moreover, the prices under the  optimal solution obtained by the proposed algorithms are lower than the SFpark solution for the entire trajectory, and the algorithms both reach the target occupancy while SFpark is far from it. The third and fourth plots of Figure \ref{fig:beachst_fo_all} show the effect of the negative price elasticity on the occupancy; an increased price causes a decreased occupancy. An interesting observation is that for Algorithm \ref{alg:main-fo}, a larger choice of $n$, and consequently a smaller choice of $T$, allows for convergence closer to the optimal price, but for Algorithm \ref{alg:main-zo}, a smaller choice of $n$, and consequently, a larger choice of $T$, allows for quicker (and with lower variance) convergence to the optimal price. This is due to the randomness in the query direction for the gradient estimator used in Algorithm \ref{alg:main-zo}, meaning that a larger $T$ is needed to converge quickly to the optimal solution. This suggests that in the more realistic case of zero order feedback, the institution should make more price announcements.

\paragraph{Redistributing Parking Demand.}
In this semi-synthetic experiment, we set $\nu=1\text{e-}3$ and take $\X=[-3.5,4.5]$ since the base distribution for these blocks has a nominal price of \$$3.50$. We also use the estimated $\georate$ and $A$ values (described in more detail in Appendix~\ref{app:sfpark-redistribute}). We run Algorithms~\ref{alg:main-zo} and \ref{alg:main-fo} (using parameters as dictated by the corresponding sample complexity theorems) for a collection of blocks during the time period $900$--$1200$ in a highly mixed use area (i.e., with tourist attractions, a residential building, restaurants and other businesses). The results are depicted in Figure~\ref{fig:bar}. 

 Hawthorne ST 0 is a very high demand street;   the occupancy is around $90$\% on average during the initial distribution and remains high for SFpark (cf.~center, Figure~\ref{fig:bar}). The performatively optimal point, on the other hand, reduces this occupancy to within the target range $60$--$80$\% for both the first and zeroth order methods. This occupancy can be seen as being redistributed to the Folsom ST 500-600 block, as depicted in Figure~\ref{fig:bar} (center) for our proposed methods: the SFpark occupancy is much below the $70$\% target average for these blocks, while both the decision-dependent algorithms lead to occupancy at the target average. 
Interestingly, this also comes at a lower price (not just on average, but for each block) than SFpark.

Hawthorne ST 100 is an interesting case in which both our approach and SFpark do not perform well. This is because the performatively optimal price in the \emph{unconstrained case} is \$$9.50$ an hour which is well above the maximum price of \$$8$ in the constrained setting we consider. In addition, the price elasticity is positive for this block; together these facts explain the low occupancy. Potentially other control knobs available to SFpark, such as time limits, can be used in conjunction with price to manage occupancy; this is an interesting direction of future work.